\documentclass[12pt]{elsarticle}

\usepackage{lineno}
\usepackage{amsmath,amssymb,amsthm,amsfonts}
\usepackage{enumitem,hyperref,graphicx}
\usepackage{nameref}

\modulolinenumbers[5]

\makeatletter
\def\namedlabel#1#2{\begingroup
    #2%
    \def\@currentlabel{#2}%
    \phantomsection\label{#1}\endgroup}

\theoremstyle{plain}
\newtheorem{theorem}{Theorem}[section]
\newtheorem{proposition}{Proposition}[section]
\newtheorem{lemma}{Lemma}[section]

\newtheorem{example}{Example}[section]

\theoremstyle{definition}
\newtheorem{definition}{Definition}[section]
\newtheorem{remark}{Remark}[section]

\newenvironment{claim}[2]{\par\noindent\emph{Claim #1:}\space#2}{}
\newenvironment{claimproof}[2]{\par\noindent\emph{Proof of Claim #1:}\space#2}{}

\renewcommand{\H}{\mathcal{H}}
\let\varepsilon\varepsilon
\newcommand{\gph}{\operatorname{gph}}

\journal{J. Differential Equations}

\begin{document}

\begin{frontmatter}

\title{Galerkin-like Method for Integro-Differential Inclusions with applications to Volterra Sweeping processes} %% Article title

\author[mymainaddress]{Pedro P\'erez-Aros}
\ead{pperez@dim.uchile.cl}
\author[mytertiaryaddress]{Manuel Torres-Valdebenito}
\ead{manuel.torres@ug.uchile.cl}

\author[mythirdaddress]{Emilio Vilches\corref{mycorrespondingauthor}}
\cortext[mycorrespondingauthor]{Corresponding author}
\ead{emilio.vilches@uoh.cl}

\address[mymainaddress]{Departamento de Ingeniería Matemática and Centro de Modelamiento Matemático  (CNRS UMI 2807), Universidad de Chile, Santiago, Chile}
\address[mysecondaryaddress]{Instituto de Ciencias de la Ingeniería, Universidad de O'Higgins, Rancagua, Chile}
\address[mytertiaryaddress]{Departamento de Ingenier\'ia Matem\'atica, Universidad de Chile, Santiago, Chile}
%\address[mythirdaddress]{Instituto de Ciencias de la Educaci\'on, Universidad de O'Higgins, Rancagua, Chile}

%% Abstract
\begin{abstract}
In this paper,  we develop the Galerkin-like method to address first-order integro-differential inclusions. Under compactness or monotonicity conditions, we obtain new results for the existence of solutions for this class of problems, which generalize existing results in the literature and provide new insights for differential inclusions with an unbounded right-hand side. The effectiveness of the proposed approach is illustrated by presenting new existence results for nonconvex state-dependent Volterra sweeping processes, where the right-hand side is unbounded, and the classical theory of differential inclusions is not applicable. This is the first result of its kind.  

The paper concludes with an application to the existence of an optimal control problem governed by nonconvex state-dependent  Volterra sweeping processes in finite dimensions.
\end{abstract}



%% Keywords

\end{frontmatter}



\section{Introduction}
First-order differential inclusions provide a general framework for studying dynamical processes, where the velocity belongs to a set that depends on time and state. It constitutes an instantaneous description of the velocity, for which there is a well-developed theory.  We refer to \cite{MR755330,MR1189795} for an overview of the subject. However, this instantaneous description does not cover phenomena where the history of the process affects the description of the system. In this paper, we introduce a class of first-order integro-differential inclusions, where the velocity depends on an integral term, which can be understood as a history operator or a given description of the acceleration.  We develop the Galerkin-like method, introduced in \cite{JV-Galerkin}, to deal with this class of integro-differential inclusions in separable Hilbert spaces.  We show that this class generalizes first-order differential equations (see, e.g., \cite{MR755330,MR1189795}) and is general enough to include the state-dependent sweeping process perturbed by an integral term, which has applications in electrical circuits, contact mechanics and crowd motion, among other areas (see, e.g., \cite{bouach2021optimal,bouach2021nonconvex}). Moreover, we provide an existence result for an optimal control problem involving the aforementioned class of integro-differential inclusions in finite-dimensional spaces.

Let $\H$ be a separable Hilbert space and $I=[0,T]$ a nonempty interval.  The first part of this paper aims to study the following integro-differential inclusion:
\begin{equation}\label{diff-inc}
\left\{
\begin{aligned}
\dot{x}(t)&\in F(t,x(t))+\int_{0}^t g(t,s,x(s))ds & \textrm{ a.e }  t\in [0,T],\\
x(0)&=x_0.
\end{aligned}
\right.
\end{equation}
To prove the existence of solutions for \eqref{diff-inc}, following the ideas from \cite{JV-Galerkin}, we approach the latter problem by projecting the state, but not the velocity, into a  {finite-dimensional} Hilbert space. Indeed, for each $n\in \mathbb{N}$ we approach \eqref{diff-inc} by the following integro-differential inclusion:
\begin{equation*}
\left\{
\begin{aligned}
\dot{x}(t)&\in F(t,P_n(x(t)))+\int_{0}^t g(t,s,P_n(x(s)))ds & \textrm{ a.e }  t\in [0,T],\\
x(0)&=P_n(x_0),
\end{aligned}
\right.
\end{equation*}
where, given an orthonormal basis $(e_n)_{n\in \mathbb{N}}$ of $\H$ and $P_n$ is the projector from $\H$ into the linear span of $\left\{ e_1,\ldots,e_n\right\}$. By means of a fixed point argument, we prove an \emph{Approximation Principle} which states that, without any compactness or monotonicity conditions, the above problem has a solution (see Theorem \ref{existencia-finito}).  Hence, the Approximation Principle provides a starting point for the approximation of differential inclusions in infinite-dimensional spaces. Once the existence of approximate solutions has been established, the idea is to pass to the limit in the differential inclusion. To do this, we provide a Compactness Principle (see Theorem \ref{main-compactness}), which establishes that whenever the trajectories are relatively compact, it is possible to ensure the existence of solutions. The Compactness Principle is used later to establish the existence of solutions for state-dependent sweeping processes (see Section \ref{Sweeping-sec}).  Finally, we show that our approach allows us to recover classical results from differential inclusions with compactness and monotonicity assumptions (See Subsections \ref{subsection-compactness} and \ref{subsection-Lipschitz}).

 The method described above is called the \emph{Galerkin-like method}, and it was introduced in \cite{JV-Galerkin}. Hence, our results extend those obtained in  \cite{JV-Galerkin}. Moreover, under compactness or monotonicity conditions, we get new results for the existence of solutions for this class of problems, which generalize existing results in the literature and provide new insights for differential inclusions with an unbounded right-hand side (see Theorem \ref{main} and Theorem \ref{main-Lipschitz}, respectively).

The second part of this paper aims to study state-dependent sweeping processes perturbed by an integral term.  The sweeping process is a first-differential inclusion involving normal cones to moving sets. It was introduced by J.J. Moreau in a series of papers to model a problem in elasto-plasticity (see \cite{MO1,MO2,MO4,Moreau1999}). Sweeping processes with integral terms have been recently considered. We can mention \cite{MR3039208}, where the authors model the movement of sticky particles via sweeping processes with an integral term.  Then, in \cite{MR4099068}, the authors use the Moreau-Yosida regularization to obtain the existence of solutions. Next, in \cite{bouach2021nonconvex}, the authors prove the existence of solutions for sweeping processes with an integral term through a numerical method. Finally, in \cite{Vilches-2024}, the author use a fixed-point argument to obtain the existence of solutions in the prox-regular case. We refer to \cite{JV-alpha,JV-Galerkin,JV-regular,MR4421900,MR3813128,MR3956966} for existence results without an integral term.
In this second part, and under mild assumptions, we use the results of the first part of the paper and the reduction technique proposed in \cite{Thibault2003} to obtain the existence of solutions.

The third part of this paper studies the existence of solutions for an optimal control problem governed by a nonconvex state-dependent integro-differential sweeping process where the control acts in the integral term.

The paper is organized as follows. After some mathematical preliminaries, in Section \ref{hipo-sol}, we gather the hypotheses used in the paper and provide some technical lemmas used to prove the main result. Then, in Section \ref{partialGalerkin}, through a fixed point argument, we prove the existence of solutions for the integro-differential inclusion \eqref{diff-inc}.   Section \ref{Sweeping-sec} and \ref{Reduction-sp} provide the well-posedness and a reduction technique for state-dependent Volterra sweeping processes. Finally, Section \ref{Control-problem} shows the existence of solutions for a related optimal control problem governed by a nonconvex state-dependent Volterra sweeping process in finite dimensions.

\section{Mathematical preliminaries}

From now on, $\H$ denotes a separable Hilbert space whose norm is denoted by $\Vert \cdot \Vert$. The closed unit ball is denoted by $\mathbb{B}$.  The closure of a set $A\subset \H$ is denoted by $\overline{A}$.
The notation $\H_w$ stands for $\H$ equipped with the weak topology, and $x_n \rightharpoonup x$ denotes the weak convergence of $(x_n)_n$ to $x$.

\subsection{Elements of differential inclusions}
We denote by $L^1\left([0,T];\H\right)$ the space of $\H$-valued Lebesgue integrable functions defined over the interval $[0,T]$. We write $L^1_w\left([0,T];\H\right)$ to mean the space $L^1\left([0,T];\H\right)$  endowed with the weak topology. Moreover, we say that $u\in \operatorname{AC}\left([0,T];\H\right)$ if there exists $f\in L^1\left([0,T];\H\right)$ and $u_0\in \H$ such that $u(t)=u_0+\int_{0}^t f(s)ds$ for  $t\in[0,T]$. 

The following result, proved in \cite{Vilches-2024}, is an enhanced version of classical Gronwall's lemma. \begin{lemma}\label{Gronwall-2}
Let $I=[0,T]$ and $u\colon I \to \mathbb{R}$ be a nonnegative essentially bounded measurable function.  Let $K_1, K_2$ and $K_3$ be nonnegative integrable functions. If for all $t\in I$ 
$$
u(t)\leq u(0)+\int_{0}^t K_1(s)ds+\int_{0}^t  K_2(s)u(s)ds+\int_{0}^t \int_{0}^s K_3(s,\tau)u(\tau)d\tau ds.
$$
Then, for all $t\in I$, one has
$$
u(t)\leq u(0)\exp\left(\int_0^t \upsilon(s)ds\right)+\int_{0}^t K_1(s)\exp\left(\int_s^t \upsilon(\tau)d\tau\right)ds.
$$
where $\upsilon(t):= K_2(t)+\int_{0}^t K_3(t,s)ds$ for $t\in I$.
\end{lemma}

\begin{proof}
    Let us consider the following non-decreasing function
    $$
        \vartheta(t):=u(0)+\int_{0}^t K_1(s)ds+\int_{0}^t K_2(s)u(s)ds+\int_{0}^t \int_{0}^s K_3(s,\tau)u(\tau)d\tau ds.
    $$
    Then, for a.e. $t\in I$, one has
    \begin{equation*}
        \begin{aligned}
          \dot{\vartheta}(t)&=K_1(t)+K_2(t)u(t)+\int_{0}^t K_3(t,s)u(s)ds\\
           &\leq K_1(t)+(K_2(t)+\int_{0}^t K_3(t,s)ds)\vartheta(t)\\
           &= K_1(t)+\upsilon(t)\vartheta(t),
        \end{aligned}
    \end{equation*}
    which, by virtue of the classical Grönwall's inequality, implies that  
    $$
    u(t)\leq \vartheta(t)\leq \vartheta(0)\exp\left(\int_0^t \upsilon(s)ds\right)+\int_{0}^t K_1(s)\exp\left(\int_s^t \upsilon(\tau)d\tau\right)ds.
    $$
    The result follows by noting that $\vartheta(0)=u(0)$.
\end{proof}

The following lemma, proved in \cite{JV-regular}, is a  criterion for weak compactness of absolutely continuous functions in $C([0,T];\H)$.
\begin{lemma}\label{compactness} Let $I:=[0,T]$ and let $(x_n)_n$ be a sequence of absolutely continuous functions from $I$ into $\H$.  Assume that for all $n\in \mathbb{N}$
    \begin{equation}
        \label{acotamiento}
        \begin{aligned}
            \Vert \dot{x}_n(t)\Vert &\leq \psi(t) &\ \textrm{ a.e } \, t\in I,
        \end{aligned}
    \end{equation}
    where {$\psi\in L^1(I;\mathbb{R})$} and that $ x_n(0) \to x_0$ as $n\to +\infty$. Then, there exists a subsequence $(x_{n_k})_k$ of $(x_n)_n$ and an absolutely continuous function $x(\cdot)$ such that
    \begin{enumerate}[label=(\roman{*})]
        \item\label{comp-i}  $x_{n_k}(t)\rightharpoonup x(t)$ in $\H$ as $k\to +\infty$ for all $t\in I$;
        \item\label{comp-ii} $x_{n_k}\rightharpoonup x$ in $L^1\left(I;\H\right)$ as $k\to +\infty$;
        \item\label{comp-iii} $\dot{x}_{n_k}\rightharpoonup \dot{x}$ in $L^1\left(I;\H\right)$ as $k\to +\infty$;
        \item\label{comp-iv} $\Vert \dot{x}(t)\Vert \leq \psi(t)$ a.e. $t\in I$.
    \end{enumerate}
\end{lemma}
\subsection{Elements of nonlinear analysis}
Let $\left(e_n\right)_{n\in \mathbb{N}}$ be an orthonormal basis of $\H$. For every $n\in \mathbb{N}$ we consider the linear operator $P_n$ from $\H$ into $\operatorname{span}\left\{ e_1,\ldots,e_n \right\}$ defined as
\begin{equation}\label{projector}
    P_{n}\left(\sum_{k=1}^{\infty}\left\langle x,e_{k} \right\rangle e_k\right)=\sum_{k=1}^n \left\langle x,e_{k} \right\rangle e_k.
\end{equation}
The following lemma summarizes the main properties of the linear operator $P_n$ (we refer to \cite{JV-Galerkin} for its proof).

\begin{lemma}\label{proyecciones} Let  $\left(e_n\right)_{n\in \mathbb{N}}$ be an orthonormal basis of a separable Hilbert space $\H$. Then, 
    \begin{enumerate}[label=(\roman{*})]
        \item $\Vert P_n(x)\Vert \leq \Vert x\Vert$ for all $x\in \H$;
        \item $\left\langle P_n(x),x-P_n(x)\right\rangle=0$ for all $x\in \H$;
        \item $P_n(x)\to x$ as $n\to +\infty$ for all $x\in \H$;
        \item\label{Pro3}  $P_{\theta(n)}(x_{\theta(n)})\rightharpoonup x$ for any $\theta\colon \mathbb{N}\to \mathbb{N}$ increasing  and  $x_{\theta(n)}\rightharpoonup x$;
        \item if $B\subset \H$ is relatively compact then $\sup_{x\in B}\Vert x-P_n(x)\Vert \to 0$ as $n\to+\infty$.
    \end{enumerate}
\end{lemma}

Let $A$ be a bounded subset of $\H$. We define the \emph{Kuratowski measure of non-compactness of $A$}, $\alpha(A)$, as
\begin{equation*}
    \alpha(A)=\inf\{d>0\colon A \textrm{ admits a finite cover by sets of diameter }\leq d\},
\end{equation*}
and the \emph{Hausdorff measure of non-compactness of} $A$, $\beta(A)$, as
\begin{equation*}
    \beta(A)=\inf\{r>0\colon A \textrm{ can be covered by finitely many balls of radius } r\}.
\end{equation*}
For convenience we define $\alpha(A) = \beta(A) =+\infty$, whenever $A$ is unbounded. 
In Hilbert spaces, the relation between these two concepts  is given by the inequality: 
\begin{equation}
    \label{equivalentes}
    \sqrt{2}\beta(A)\leq \alpha(A)\leq 2\beta(A) \textrm{ for } A\subset \H \textrm{ bounded.}
\end{equation}

The following proposition summarizes the main properties of Kuratowski and Hausdorff measures of non-compactness (see \cite[Section~9.2]{MR1189795} ).
\begin{proposition}\label{Kura}
    Let $\H$ be an infinite dimensional Hilbert space and $B,B_1$ and $B_2$ be subsets of $\H$. Let $\gamma$ be either the Kuratowski or the Hausdorff measures of non-compactness. Then,
    \begin{enumerate}[label=(\roman{*})]
        \item $\gamma(B)=0$ if and only if $\overline{B}$ is compact;
        \item $\gamma(\lambda B)=|\lambda|\gamma(B)$ for every $\lambda\in \mathbb{R}$;
        \item  $\gamma(B_1+B_2)\leq \gamma(B_1)+\gamma(B_2)$;
        \item $B_1\subset B_2$ implies $\gamma(B_1)\leq \gamma(B_2)$;
        \item $\gamma(\operatorname{conv}B)=\gamma(B)$;
        \item $\gamma(\bar{B})=\gamma(B)$.
    \end{enumerate}
\end{proposition}

The following lemma (see \cite[Proposition~9.3]{MR1189795}) is a useful rule for the interchange of $\beta$ and integration.  
\begin{lemma}\label{2.3}
     Let $(v_n)$ be a sequence of measurable functions $v_n\colon [0,T]\to \H$ such that $\sup_{n}\Vert v_n(t)\Vert\leq \psi(t)$ a.e. $t\in [0,T]$, where  $\psi$ is integrable. Then
    \begin{equation*}
        \beta\left(\left\{\int_{t}^{t+h}v_n(s)ds\colon n\in \mathbb{N} \right\}\right)\leq \int_{t}^{t+h}\beta\left(\{v_n(s)\colon n\in \mathbb{N}\}\right)ds,
    \end{equation*}
    for $0\leq t<t+h\leq T$.
\end{lemma}

The following result is due to Gohberg - Goldenstein - Markus (see, e.g., \cite[Theorem~3.9]{MR3587794}).

\begin{lemma}
    \label{GGM} 
    Let $\H$ be a separable Hilbert space and $A$ be a bounded subset of $\H$. Then 
    $$
        \frac{1}{a}\inf_n \sup_{x\in A}\Vert (I-P_n)(x)\Vert \leq \beta(A)\leq \inf_n \sup_{x\in A}\Vert (I-P_n)(x)\Vert,
    $$
    where $P_n$ is the projector defined in \eqref{projector} and $a=\limsup_{n\to +\infty}\Vert I-P_n\Vert$ is a constant.
\end{lemma}

\begin{lemma}\label{Lemma-Constante}
    There exists $C >0$ such that for any  bounded sequence  $(x_n)$ in a separable Hilbert space $\H$ the following inequality holds:
    \begin{equation}
        \label{Constante}
        \gamma(\{ P_k x_k:  k\in \mathbb{N}\})\leq C\gamma(\{ x_k: k 
        \in \mathbb{N}\}),
    \end{equation}
    where $\gamma$ is the Kuratowski or the Hausdorff measure of non-compactness.
\end{lemma}

\begin{proof} 
    According to \eqref{equivalentes}, it is enough to prove the inequality \eqref{Constante} for the Hausdorff measures of non-compactness  $\beta$. Indeed, due to Lemma \ref{GGM}, it follows that
    \begin{equation*}
        \begin{aligned}
            \beta(\{ P_k x_k\colon k\in \mathbb{N}\})
            &\leq \inf_{n}\sup_{k\in \mathbb{N}}\Vert (I-P_n)(P_k x_k)\Vert \\
            &\leq \inf_{n}\sup_{k\in \mathbb{N}}\Vert P_k(I-P_n)(x_k)\Vert \\
            &\leq \inf_{n}\sup_{k\in \mathbb{N}}\Vert (I-P_n)(x_k)\Vert\\
            &\leq a\beta(\{x_k\colon k\in \mathbb{N}\}),
        \end{aligned}
    \end{equation*}
    where $a$ is the constant given by  Lemma \ref{GGM}.
\end{proof}
\begin{lemma}\label{Lemma-measure-Lipschitz}
    Let $g\colon \H \to \H$  be a function and $B\subset \H$ be a bounded set such that there exists $\kappa\geq 0$  such that
    $$
    \Vert g(x)-g(y)\Vert \leq \kappa \Vert x-y\Vert \textrm{ for all } x,y\in B.
    $$
    Then, for all  $A\subset B$ the following inequality holds 
    \begin{equation}
        \label{C-Lipschitz}
        \gamma(g(A))\leq \sqrt{2}\kappa \gamma(A),
    \end{equation}
    where $\gamma$ is the Kuratowski or the Hausdorff measures of non-compactness.
\end{lemma}
\begin{proof}  Let $A\subset B$.
First, let us show that 
\begin{equation}
        \label{C-Lipschitz2}
        \alpha(g(A))\leq \kappa \alpha(A).
    \end{equation}
Indeed, suppose that $\alpha(A)=d$. Then, given $\varepsilon>0$, we can find sets $S_j$ such that $A=\cup_{j=1}^m S_j$, with $\operatorname{diam}(S_j)\leq d+\varepsilon$. Thus, $g(A)=\cup_{j=1}^m g(S_j)$. Moreover, since $S_j\subset B$ and $g$ is $\kappa$-Lipschitz on $B$, we obtain that
    $\operatorname{diam}(g(S_j))\leq \kappa (d+\varepsilon)$. Therefore, $g(A)$  admits a finite cover by sets $g(S_j)$ of diameter less or equal to $\kappa (d + \varepsilon)$, which means that  $\alpha (g(A)) \leq \kappa (d + \varepsilon)$. Since $\varepsilon$ is arbitrary, we conclude that $\alpha(g(A))\leq \kappa d$, which proves \eqref{C-Lipschitz2}. Finally, inequality \eqref{C-Lipschitz} for the Hausdorff measure of non-compactness $\beta$ follows from inequalities  \eqref{equivalentes} and  \eqref{C-Lipschitz2}.
\end{proof}

\subsection{Elements of Measure Theory}
Consider a topological space $Y$ and a set-valued mapping $M \colon [0,T] \to Y$.  We say that $M$ is measurable if the set $$M^{-1}(U):=\{ t \in [0,T] : M(t) \cap U \neq \emptyset\} $$ is a (Lebesgue) measurable set for  every open set $U\subset Y$.

\begin{lemma}\label{beta-measurability}
    Let $M: [0,T] \rightrightarrows \H$ be a measurable set-valued mapping. Then, the function $ t \mapsto \beta(M(t)) $ is measurable. Particularly, for any sequence $(v_n)_n$ of measurable functions from $[0,T]$ into $\H$  the  function $t \mapsto \beta(\{v_n(t)\colon n\in \mathbb{N} \})$ is measurable.  
 %   where $\gamma$ is either the Kuratowski or the Hausdorff measures of non-compactness.
\end{lemma}
\begin{proof}
Let us notice that by Proposition \ref{Kura},  $\beta(M(t)) = \beta(\overline{M}(t))$ for all $t\in [0,T]$.   Hence, we can assume without loss of generality that $M(t)$ is closed.  Let us consider a dense set $\{ x_k\}_{k=1}^\infty \subset \H$ and $r\in \mathbb{R}$. Let us denote $U_n:= \bigcup_{ k=1}^n \mathbb{B}_{r}(x_k)$. Hence, by \cite[Theorem 8.1.4]{Aubin_Frankowska_2009_book},  we have that the set  $M^{ -1} \left( \H \setminus U_n \right)$ is a (Lebesgue) measurable set. Finally, we have that 
    \begin{align*}
        \{ t\in [0,T]:  \beta(M(t)) < r\} & = \bigcup_{ n=1 }^\infty \left\{ t: M(t) \subset U_n \right\} = \bigcup_{ n=1 }^\infty [0,T]\setminus  M^{ -1} \left( \H \backslash U_n \right),
    \end{align*}
   which shows the measurability of $t \mapsto \beta(M(t))$. The last assertion follows by noticing that   $M(t):= \{v_n(t): n\in \mathbb{N} \}  $ is a measurable set-valued mapping.  
\end{proof}

\subsection{Tools from Variational Analysis}
A vector $h\in \H$ belongs to the \emph{Clarke tangent cone} $T(S;x)$ whenever for every sequence $(x_n)_n$ in $S$ converging to $x$ and every sequence of positive numbers $(t_n)_n$ converging to $0$, there exists some sequence $(h_n)_n$ in $\H$ converging to $h$ such that $x_n+t_nh_n\in S$ for all $n\in \mathbb{N}$. This cone is closed and convex and its negative polar $N(S;x)$ is the \emph{Clarke normal cone} to $S$ at $x\in S$, that is,
\begin{equation*}
N\left(S;x\right):=\left\{v\in \H\colon \left\langle v,h\right\rangle \leq 0 \quad  \forall h\in T(S;x)\right\}.
\end{equation*}
As usual, $N(S;x)=\emptyset$ if $x\notin S$. Through that normal cone, the Clarke subdifferential of a function $f\colon \H\to \mathbb{R}\cup\{+\infty\}$ is defined by
\begin{equation*}
\partial f(x):=\left\{v\in \H\colon (v,-1)\in N\left(\operatorname{epi}f,(x,f(x))\right)\right\},
\end{equation*}
where $\operatorname{epi}f:=\left\{(x,r)\in \H\times \mathbb{R}\colon f(x)\leq r\right\}$ is the epigraph of $f$. When the function $f$ is finite and locally Lipschitzian around $x$, the Clarke subdifferential is characterized (see \cite{Clarke1998}) in the following simple and amenable way
\begin{equation*}
\partial f(x)=\left\{v\in \H\colon \left\langle v,h\right\rangle \leq f^{\circ}(x;h) \textrm{ for all } h\in \H\right\},
\end{equation*}
where
\begin{equation*}
f^{\circ}(x;h):=\limsup_{(t,y)\to (0^+,x)}t^{-1}\left[f(y+th)-f(y)\right],
\end{equation*}
is the \emph{generalized directional derivative} of the locally Lipschitzian function $f$ at $x$ in the direction $h\in \H$.  The function $f^{\circ}(x;\cdot)$ is in fact the support of $\partial f(x)$. That characterization easily yields that the Clarke subdifferential of any locally Lipschitzian function has the important property of upper semicontinuity from $\H$ into $\H_w$.

For $x\in \H$ and $S\subset \H$ the distance function to the set $S$ at $x\in \H$ is defined by $d_{S}(x):=\inf_{y\in S}\Vert x-y\Vert$. We denote $\operatorname{Proj}_{S}(x)$ the set (possibly empty)
\begin{equation*}
\operatorname{Proj}_{S}(x):=\left\{y\in S\colon d_{S}(x)=\Vert x-y\Vert\right\}.
\end{equation*}
The equality (see, e.g.,  \cite[p.~85]{Clarke1998})
\begin{equation}\label{eq.13}
\begin{aligned}
N\left(S;x\right)&=\overline{\mathbb{R}_+\partial d_S(x)}^{\ast} & \textrm{ for } x\in S,
\end{aligned}
\end{equation}
gives an expression of the Clarke normal cone in terms of the distance function. As usual, it will be convenient to write $\partial d(x,S)$ in place of $\partial d\left(\cdot,S\right)(x)$.

The next result provides the measurability of the Clarke subdifferential for measurable and locally Lipschitz continuous integrands, which will be used in Section \ref{Sweeping-sec}. For a function $f\colon  [0,T] \times\H\times \H \to \mathbb{R}$,   we  denote   $f^\circ_x(t, x, y;h) := (f(t, \cdot, y))^\circ (x;h)$ and 
\begin{equation*}
    \partial_x f(t,x,y) := \left\{  v\in \H : \langle v,h \rangle \leq f^\circ_x(t, x, y;h)  \text{ for all }h \in \H  \right\}.
\end{equation*}
\begin{lemma}\label{Clarke-measurable}
    Let $f\colon [0,T] \times\H\times \H \to \mathbb{R}$ be  measurable with respect to $t\in [0,T]$ and locally Lipschitz with respect to $(x,y) \in \H\times \H$.  Then, the map $(t,x,y,h) \to f^\circ_x (t,x,y;h)$ is measurable, usc with respect to $(x,y)$ and continuous with respect to $h$. Consequently,  the mapping
    $t\rightrightarrows \gph \partial_x f(t,\cdot,\cdot) $  is measurable. Moreover,  the set  $\gph \partial_x f(t,\cdot,\cdot)$ is closed on $\H\times \H_w$ for all $t\in [0,T]$.
\end{lemma}
\begin{proof}
First, it is well-known that $ f^\circ_x (t,x,y;h)$ is usc with respect to $(x,y)$ and continuous with respect to $h$.  We claim that  the function 
    $(t,x,y,h) \mapsto f_x^{\circ}(t,x,y;h)  $ is measurable. Indeed,  consider a sequence $\varepsilon_k \to 0^+$, and a countable dense set $B \subset \mathbb{B}$. Hence,  one has
\begin{equation*}
f_x^{\circ}(t,x,y;h) = \inf_{ k \in \mathbb{N}} \sup_{ \substack{ w \in \varepsilon_k  B\\ s \in (0, \varepsilon_k) \cap \mathbb{Q}}} s^{-1}\left[f(t,x + w+sh,y)-f(t,x+w,y)\right], 
\end{equation*}
which shows the measurability of $(t,x,y,h) \mapsto f_x^{\circ}(t,x,y;h)$. Finally, 
    \begin{align*}
    \gph \partial_x f(t,\cdot) & = \bigcap_{ n \in \mathbb{N} } \left\{ (t, x,y,v) :  \left\langle v  ,h_n\right\rangle \leq f_x^{\circ}(t,x,y;h_n)    \right\}, 
    \end{align*}
    which proves the measurability of $t\rightrightarrows \gph \partial f(t,\cdot,\cdot)$ (see, e.g., \cite[Theorem 8.1.4]{Aubin_Frankowska_2009_book}).
\end{proof}

To deal with Volterra sweeping processes, we recall the definition of the class of positively $\alpha$-far sets, introduced in \cite{Haddad2009} and widely studied in \cite{JV-alpha}. 
\begin{definition} Let $\alpha\in ]0,1]$ and $\rho\in ]0,+\infty]$. Let $S$ be a nonempty closed subset of $\H$ with $S\neq \H$. We say that the Clarke subdifferential of the distance function $d(\cdot,S)$ keeps the origin $\alpha$-far-off on the open tube around $S$, $U_{\rho}(S):=\{x\in \H\colon 0<d(x,S)<\rho\}$, provided
\begin{equation}\label{13}
0<\alpha\leq \inf_{x\in U_{\rho}(S)}d(0,\partial d(\cdot,S)(x)).
\end{equation}
Moreover, if $E$ is a given nonempty set, we say that the family $(S(t))_{t\in E}$ is positively $\alpha$-far if every $S(t)$ satisfies \eqref{13} with the same $\alpha\in ]0,1]$ and $\rho>0$.
\end{definition}
This notion strictly includes  the notion of uniformly subsmooth sets and the notion of uniformly prox-regular sets (see \cite{JV-alpha}).

\begin{definition}Let $S$ be a closed subset of $\H$. We say that $S$ is \emph{uniformly subsmooth}, if for every $\varepsilon>0$ there exists $\delta>0$, such that
\begin{equation}\label{14}
  \left\langle x_1^*-x_2^*,x_1-x_2\right\rangle \geq -\varepsilon \Vert x_1-x_2\Vert
\end{equation}
holds for all $x_1,x_2\in S$,  satisfying $\Vert x_1-x_2\Vert <\delta$ and all $x_i^*\in N\left(S;x_i\right)\cap \mathbb{B}$ for $i=1,2$. Also, if $E$ is a given nonempty set, we say that the family $\left(S(t)\right)_{t\in E}$ is \emph{equi-uniformly subsmooth}, if for every $\varepsilon>0$, there exists $\delta>0$ such that \eqref{14} holds for each $t\in E$ and all $x_1,x_2\in S(t)$ satisfying $\Vert x_1-x_2\Vert <\delta$ and all $x_i^*\in N\left(S(t);x_i\right)\cap \mathbb{B}$ for $i=1,2$.
\end{definition}
It is worth emphasizing that the above definitions include the class of convex and uniformly prox-regular sets, which are common in the study of the sweeping process (see, e.g., \cite{JV-alpha,JV-regular}). It is well-known that uniformly subsmooth sets are normally regular (see, e.g.,  \cite{MR2115366}), which implies that for any  uniformly subsmooth set $S$ one has
\begin{equation}\label{normally-dist}
N(S;x)\cap \mathbb{B}=\partial d_S(x) \textrm{ for all } x\in S.
\end{equation}

\section{Technical assumptions and technical lemmas}\label{hipo-sol}
\begin{enumerate}
    \item[\namedlabel{HF}{$(\mathcal{H}^F)$}]   The set-valued map $F\colon [0,T]\times \H\rightrightarrows \H$  has nonempty, closed and convex values.
    \begin{enumerate}
        \item[\namedlabel{H1F}{$(\mathcal{H}_1^F)$}] The map $t\rightrightarrows \operatorname{gph}F(t,\cdot)$ is measurable, that is, for every (norm) open set $U\subset \H\times \H$, the set  $\{t: \operatorname{gph}F(t,\cdot)\cap U \neq \emptyset\}$ is Lebesgue measurable.
        \item[\namedlabel{H2F}{$(\mathcal{H}_2^F)$}] For a.e.  $t\in [0,T]$,  $\operatorname{gph}F(t,\cdot)$ is closed on $\H\times \H_w$.
        \item[\namedlabel{H3F}{$(\mathcal{H}_3^F)$}] There exist two nonnegative integrable functions $c$ and $d$ such that
        \begin{equation*}
            \begin{aligned}
                d\left(0,F(t,v)\right):=\inf\{\Vert w\Vert \colon w\in F(t,v)\}\leq c(t)\Vert v\Vert +d(t),
            \end{aligned}
        \end{equation*}
        for all $v\in \H$ and a.e. $t\in [0,T]$.
    \end{enumerate}
\end{enumerate}

\begin{enumerate}
    \item[\namedlabel{H4F}{$(\mathcal{H}_4^F)$}]   
For all $r>0$, there exists an integrable function $k_r\colon [0,T]\to \mathbb{R}_+$ such that for a.e. $t\in [0,T]$ and $A\subset r\mathbb{B}$, one has 
    $$
    \gamma(F(t,A))\leq k_r(t)\gamma(A),
    $$
 where $\gamma$ is either the Kuratowski or the Hausdorff measure of noncompactness.
\end{enumerate}

\begin{enumerate}
    \item[\namedlabel{H5F}{$(\mathcal{H}_5^F)$}]  For all $r>0$, there exists an integrable function $\tilde{k}_r\colon [0,T]\to \mathbb{R}$ such that for a.e. $t\in [0,T]$ and all $x_1, x_2\in r\mathbb{B}$ 
    $$
    \langle v_1-v_2,x_1-x_2\rangle \leq \tilde{k}_r(t)\Vert x_1-x_2\Vert^2 \textrm{ for all } v_1\in F(t,x_1),  v_2\in F(t,x_2).
    $$
\end{enumerate}

\begin{enumerate}
    \item[\namedlabel{Hg}{$(\mathcal{H}^g)$}]  The function $g\colon [0,T]\times[0,T]\times \H\to \H$ satisfies
    \begin{enumerate}
        %\item[\namedlabel{H1g}{$(\mathcal{H}_1^g)$}] For each $v\in \H$, $g(\cdot,v)$ is measurable.
        \item For each $v\in \H$, the map $(t,s) \mapsto g(t,s,v)$ is measurable.
        \item For all $r>0$, there exists an integrable function $\mu_{r}\colon [0,T]\to \mathbb{R}_+$ such that for all $(t,s)\in D$ 
        $$
        \Vert g(t,s,x)-g(t,s,y)\Vert \leq \mu_{r}(t)\Vert x-y\Vert \textrm{ for all } x,y\in r\mathbb{B}.
        $$
        Here $D:=\{(t,s)\in [0,T]\times [0,T]\colon s\leq t\}$. 
        \item There exists a nonnegative integrable function $\sigma\colon D\to \mathbb{R}$ such that 
        $$
        \Vert g(t,s,x)\Vert \leq \sigma(t,s)(1+\Vert x\Vert) \textrm{ for all } (t,s)\in D \textrm{ and } x\in \H. 
        $$
    \end{enumerate}
\end{enumerate}

\begin{remark} %It is worth pointing out the following observations.
 %   \begin{itemize}
  %\item[(i)] 
   Let us recall that $\H_w$ is a Suslin space, and it is not difficult to see that  $\mathcal{B}(\H) =\mathcal{B}(\H_w)$. Furthermore, according to \cite[Proposition 1.48]{PapaHandbook-1},  $$\mathcal{B}(\H \times \H_w) =\mathcal{B}(\H) \otimes \mathcal{B}(\H_w).$$
  %  \end{itemize}
   % \item[(ii)] Let $F\colon [0,T]\times \H \rightrightarrows \H$ be a set-valued mapping with nonempty, closed, convex and bounded values. Then, assumption \ref{H4F} holds if for all $r>0$, there exists an integrable function $k_r\colon [0,T]\to \H$ such that for a.e. $t\in [0,T]$, one has 
    %$$
    %\operatorname{Haus}(F(t,x),F(t,y))\leq k_r(t)\Vert x-y\Vert \textrm{ for all } x,y\in r\mathbb{B},
    %$$
   % where $\operatorname{Haus}(A,B)$ denotes the Hausdorff distance between $A$ and $B$.
   % \item[(iii)] Assumption \ref{H5F} is a generalization of the so-called uniform one-sided Lipschitz condition (see, e.g., \cite{MR1166177}). 
\end{remark}
The following result provides a selection result that will be used in the proof of the Approximation Principle (Theorem \ref{existencia-finito}).
\begin{lemma}\label{Lemma_exis}
     Let $F\colon [0,T]\times \H\rightrightarrows \H$  be a set-valued mapping satisfying \ref{H1F} and \ref{H2F}. Then, for any $ u\colon [0,T] \to \H$ measurable, the set-valued mapping $M(t) := F(t,u(t)) $ is measurable.  Moreover, $M$ admits a measurable selection $v\colon [0,T]\to \H$ such that
     $$
     d\left(0,F(t,u(t) )\right) =\| v(t)\| \textrm{ for a.e. } t\in [0,T].
     $$
\end{lemma}
\begin{proof}
    First, by virtue of \cite[Theorem 1.35]{PapaHandbook-1} (or  \cite[Theorem III.30]{Castaing}), we have that $\gph F \in \mathcal{L}([0,T]) \otimes \mathcal{B}(\H \times \H)$. Furthermore, 
    \begin{align*}%\label{eq00_meas}
    \gph M = \{ (t,x) : (t,u(t), x) \in \gph F\} = \varphi^{-1} (\gph F),
    \end{align*}
    where $\varphi \colon  [0,T] \times \H \to  [0,T]\times \H \times  \H$ is given by $\varphi(t,x):= (t,u(t), x)$, which is a measurable function. Consequently, the measurability of $\gph M$ follows from the measurability of $\varphi$ and $\gph F$. Since $\gph M \in \mathcal{L}([0,T]) \otimes \mathcal{B}$ and $M$ has closed values, we can apply again \cite[Theorem 1.35]{PapaHandbook-1} (or  \cite[Theorem III.30]{Castaing})  to conclude that $M$ is measurable with nonempty closed and convex values. Then, due to \cite[Corollary 8.2.13]{Aubin_Frankowska_2009_book},  $v(t):=\operatorname{proj}_{ M(t)}(0) \in M(t)$ is well-defined and measurable. Hence, it follows  that $ d\left(0,F(t,u(t) )\right) =\| v(t)\|$ for a.e. $t\in [0,T].$
\end{proof}

It is worth emphasizing that we are not assuming the boundedness for the set-valued mapping defined in \ref{HF}. Hence, the following lemma will be used in the proof of Theorem \ref{existencia-finito}.
\begin{lemma}\label{H123} 
    Assume that \ref{H1F}, \ref{H2F} and \ref{H3F} hold and let $r\colon [0,T]\to \mathbb{R}_+$ be a continuous function. Then, the set-valued map $G\colon [0,T]\times \H \rightrightarrows \H$ defined by
    \begin{equation}\label{DefG}
        \begin{aligned}
            G(t,x)&:=F(t,p_{r(t)}(x))\cap \left(c(t)\Vert p_{r(t)}(x)\Vert+d(t)\right)\mathbb{B},
        \end{aligned}
    \end{equation}
    where $p_{r(t)}(x):=\operatorname{proj}_{r(t)\mathbb{B}}(x)$,  satisfies:
    \begin{enumerate}[label=(\roman{*})]
        \item\label{mea0} $G(t,x)$ is nonempty, closed and convex for all $(t,x)\in [0,T]\times \H$;
        \item  the map $t\rightrightarrows  \operatorname{gph} G(t,\cdot)$ is measurable;
        \item\label{mea2} for a.e.  $t\in [0,T]$, $\gph G(t,\cdot)$ is closed on $\H\times \H_w$;
        \item\label{mea3} for all $x\in \H$ and a.e. $t\in [0,T]$
        \begin{equation*}
        \Vert G(t,x)\Vert:=\sup\{\Vert w\Vert \colon w\in G(t,x)\}\leq c(t)r(t)+d(t).
        \end{equation*}
    \end{enumerate}
\end{lemma}

\begin{proof}
    \ref{mea0} is direct. \ref{mea2} follows from \ref{H2F} and \cite[Theorems~17.25 and 17.32]{Aliprantis}. Also, due to \ref{H3F}, we have
    \begin{equation*}
        \begin{aligned}
            \Vert G(t,x)\Vert =\sup_{w\in G(t,x)}\Vert w\Vert \leq c(t)\Vert p_{r(t)}(x)\Vert+d(t)\leq c(t)r(t)+d(t),
            \end{aligned}
    \end{equation*}
    which proves \ref{mea3}. Thus, by virtue of \ref{mea0} and \ref{mea3}, $G$ takes weakly compact and convex values. Now, let us observe that $\gph G(t,\cdot)=\gph F(t,\cdot)\cap \gph H(t,\cdot)$, where $H(t,x):=(c(t)\Vert p_{r(t)}(x)\Vert +d(t))\mathbb{B}$. It is easy to see that $t\mapsto \gph H(t,\cdot)$ is measurable, so by \cite[Theorem 8.2.4]{Aubin_Frankowska_2009_book}, we get the measurability of $t\mapsto \gph G(t,\cdot )$.
\end{proof}

\section{Galerkin-like method for integro-differential inclusions}\label{partialGalerkin}
In this section, we study the existence of solutions to the following integrodifferential inclusion:
\begin{equation}\label{Problema}
    \left\{
    \begin{aligned}
        \dot{x}(t)&\in F(t,x(t))+\int_{0}^t g(t,s,x(s))ds & \textrm{ a.e. } t\in [0,T];\\
        x(0)&=x_0,
    \end{aligned}
    \right.
\end{equation}
where $F\colon [0,T]\times \H \rightrightarrows  \H$ is a set-valued map with nonempty, closed, and convex values and $g\colon [0,T]\times[0,T] \times\H \to \H$ is a given function. For every $n\in \mathbb{N}$ let us consider the following integro-differential inclusion:
\begin{equation}\label{Problema-n}
    \left\{
    \begin{aligned}
        \dot{x}(t)  &   \in F(t,P_n(x(t)))+\int_{0}^t g(t,s,P_n(x(s)))ds 
        & \textrm{ a.e. } t\in [0,T];\\
        x(0)      &   =P_n(x_0),
    \end{aligned}
    \right.
\end{equation}
where $P_n\colon \H \to \operatorname{span}\left\{ e_1,\ldots,e_n\right\}$ is the linear operator defined in  Lemma \ref{proyecciones}.
The next proposition asserts the existence of solutions for the approximate problem \eqref{Problema-n}.  We will call this result the \emph{Approximation Principle}, since it provides a methodology to approximate the solutions of integro-differential inclusions. In Subsections  \ref{subsection-compactness} and \ref{subsection-Lipschitz}, respectively, we will see that the Approximation Principle, together with a condition of compactness or monotonicity, allows us to obtain the existence of solutions for the problem \eqref{Problema}.

\begin{theorem}[Approximation Principle for integro-differential inclusions]
    \label{existencia-finito} 
    Assume that \ref{HF} and \ref{Hg} hold. Then, for each $n\in \mathbb{N}$ the problem \eqref{Problema-n} admits at least one solution $x_n\in \operatorname{AC}\left([0,T];\H\right)$. Moreover,
    \begin{equation}\label{cota1}
        \begin{aligned}
            \Vert x_n(t)\Vert &\leq r(t) & \textrm{ for all } t\in [0,T],
        \end{aligned}
    \end{equation}
    where for all $t\in [0,T]$ 
    $$ 
        r(t):=\Vert x_0\Vert \exp\left(\int_0^t \eta(s)ds\right)+\int_{0}^t\varepsilon(s)\exp\left(\int_s^t \eta(\tau)d\tau\right)ds,
    $$
    where $\eta(t):= c(t)+\int_{0}^t\sigma(t,s)ds$ and $\varepsilon(t):=d(t)+\int_{0}^t\sigma(t,s)ds$ for all $t\in [0,T]$.
    Moreover, one has
    \begin{equation}\label{cota2}
        \begin{aligned}
            \Vert \dot{x}_n(t)\Vert &\leq \psi(t):= \eta(t)r(t)+\varepsilon(t) \textrm{ for a.e. } t\in [0,T].
        \end{aligned}
    \end{equation}  
\end{theorem}

\begin{proof} Let us consider $G(t,x)$ as the mapping defined in \eqref{DefG}, where $r(t)$ is defined as in the statement of the theorem.     Then, due to  Lemma \ref{H123}, $G$ satisfies \ref{H1F}, \ref{H2F} and for a.e. $t\in[0,T]$
    \begin{equation*}%\label{cota-F}
        \Vert G(t,x)\Vert:=\sup\{\Vert w\Vert \colon w\in G(t,x)\} \leq c(t)r(t)+d(t) \textrm{ for all } x\in \H.
    \end{equation*}
Consider the following differential inclusion:
    \begin{equation}\label{Problema-simp}
        \left\{
        \begin{aligned}
            \dot{x}(t)&\in G(t,P_n(x(t)))+\int_{0}^t g(t,s,p_{r(s)}(P_n(x(s))))ds &  \textrm{ a.e. } t\in [0,T];\\
            x(0)&=P_n(x_0).
        \end{aligned}
        \right.
    \end{equation}
    Let $K\subset L^1\left([0,T];\H\right)$ be defined by
    \begin{equation*}
        K:=\left\{f\in L^1\left([0,T];\H\right)\colon \Vert f(t)\Vert \leq \psi(t) \textrm{ a.e. } t\in [0,T] \right\},
    \end{equation*}
    where $\psi$ is defined by \eqref{cota2}. This set is nonempty, closed, and convex. In addition, since $\psi$ is integrable, the set $K$ is bounded and uniformly integrable. Hence, it is compact in $L^1_{w}\left([0,T];\H\right)$ (see \cite[Theorem~2.3.24]{Papa2005}). Since  $L^1\left([0,T];\H\right)$ is separable, we also note that $K$, endowed with the relative $L^1_{w}\left([0,T];\H\right)$ topology is a metric space (see \cite[Theorem~V.6.3]{Dunford-Sch}). Consider the map $\mathcal{F}_n\colon K \rightrightarrows L^{1}\left([0,T];\H\right)$ defined for $f\in K$ as those $v\in L^{1}\left([0,T];\H\right)$ such that for a.e. $t\in [0,T]$
    \begin{equation*}
        v(t)\in G(t,P_n(x_0+\int_{0}^t f(s)ds))
        +\int_{0}^t g(t,s,p_{r(s)}(P_n(x_0+\int_{0}^s f(\tau)d\tau)))ds,
    \end{equation*}
 By Lemma \ref{Lemma_exis}, we conclude that $\mathcal{F}_n(f)$ is nonempty. Moreover,  $\mathcal{F}_n(f)$ is  closed and convex due to the convexity and closedness of the values of $G$. Now, let us show that  $\mathcal{F}_n(K)\subset K$.   Indeed, let $f\in K$ and $v\in \mathcal{F}_n(f)$. Then,
    \begin{equation*}
        \begin{aligned}
            \Vert v(t)\Vert 
            &
            \leq \sup\left\{\Vert w\Vert \colon w\in G(t,P_{n}(x_{0}+\int_{0}^{t} f(s)ds)) \right.\\
            &
            \qquad\qquad\qquad\left.+\int_{0}^{t} g(t,s,p_{r(s)}(P_{n}(x_{0}+\int_{0}^{s} f(\tau)d\tau))) ds\right\}\\
            &\leq
            c(t)r(t)+d(t) +\int_{0}^{t} \sigma(t,s)ds +\int_{0}^{t} \sigma(t,s)r(s)ds\\ 
            &\leq  \eta(t)r(t)+\varepsilon(t)=\psi(t),
        \end{aligned}
    \end{equation*}
    where we have used the definition of $G$ and the map $x\mapsto p_{r(t)}(x)$. \newline \noindent 
    We denote $K_w$ the set $K$ seen as a compact convex subset of $L^1_w\left([0,T];\H\right)$.
    \begin{claim}{}{
        $\mathcal{F}_n$ is upper semicontinuous from $K_w$ into $K_w$.
    }\end{claim}

    % Claim 1
    \begin{claimproof}{\hspace{-1mm}}{
        By virtue of \cite[Proposition~1.2.23]{PapaHandbook-1}, it is sufficient to prove that its $\operatorname{graph}(\mathcal{F}_n)$ is sequentially closed in $K_w\times K_w$. Indeed, let $(f_m,v_m)\in \operatorname{graph}(\mathcal{F}_n)$ with $f_m\to f$ and $v_m\to v$ in $L^1_w\left([0,T];\H\right)$ as $m\to +\infty$. We have to show that $(f,v)\in \operatorname{graph}(\mathcal{F}_n)$. To do that, let us define
        \begin{equation*}
            \begin{aligned}
                u_m(t)&:=P_n(x_0)+\int_{0}^t f_m(s)ds & \textrm{ for every } t\in [0,T].
            \end{aligned}
        \end{equation*}
        Thus, for a.e. $t\in [0,T]$,
        \begin{equation}\label{paso-limite}
            \begin{aligned}
                v_m(t)\in G(t,P_n(u_m(t)))+\int_{0}^t g(t,s,p_{r(s)}(P_n(u_m(s))))ds.
            \end{aligned}
        \end{equation}
         Also, since $f_m\in K$, we have that
        \begin{equation*}
            \begin{aligned}
              \Vert \dot{u}_m(t)\Vert &\leq \psi(t) &  \textrm{ a.e. } t\in [0,T].
          \end{aligned}
        \end{equation*}}
        Hence, due to  Lemma \ref{compactness}, there exists a subsequence of $(u_m)_m$ (without relabeling) and an absolutely continuous function $u\colon [0,T]\to \H$ such that
        \begin{equation*}
                u_m(t) \to u(t)  \textrm{ weakly  for all } t\in [0,T] \textrm{ and 
 } \dot{u}_m\to \dot{u} \textrm{ in } L^1_w\left([0,T];\H\right),
        \end{equation*}
        which implies that $\dot{u}=f$.
        Moreover, since $(u_m(t))_{m}$ is bounded for every $t\in [0,T]$, $P_n(u_m(t))\to P_n(u(t))$ for every $t\in [0,T]$.
        Consequently, by virtue of \cite[Proposition~2.3.1]{Papa2005}, \eqref{paso-limite} and the upper semicontinuity of $G$ from $\H$ into $\H_w$, for a.e. $t\in [0,T]$, one has        
        \begin{equation}
            \label{estimation001}
            \begin{aligned}
                v(t)
                &\in \overline{\operatorname{conv}}\,w\textrm{-}\limsup_{m\to +\infty}\{v_{m}(t)\}\\
                &\subset \overline{\operatorname{conv}}\,\left[G(t,P_n(u(t)))+\int_{0}^t g(t,s,p_{r(s)}(P_n(u(s))))ds\right]\\
                &=G(t,P_n(u(t)))+\int_{0}^t g(t,s,p_{r(s)}(P_n(u(s))))ds,
            \end{aligned}
        \end{equation}
        which shows that $(f,v)\in \operatorname{graph}(\mathcal{F}_n)$,  as claimed.
    \end{claimproof} \newline    
    Now, we can invoke the Kakutani-Fan-Glicksberg fixed point theorem (see \cite[Corollary~17.55]{Aliprantis}) to the set-valued map $\mathcal{F}_n\colon K_w \rightrightarrows K_w$ to deduce the existence of $\widehat{f}_n\in K$ such that $\widehat{f}_n\in \mathcal{F}_n(\widehat{f}_n)$. Then, the function $x_n\in \operatorname{AC}\left([0,T];\H\right)$ defined for every $t\in [0,T]$ as:
    \begin{equation*}
        x_n(t)=P_n(x_0)+\int_{0}^{t} \widehat{f}_{n}(s)ds,
    \end{equation*}
    is a solution for  \eqref{Problema-simp}. Moreover, $x_n\in \operatorname{AC}\left([0,T];\H\right)$ is a solution of  \eqref{Problema-n}. Indeed, for a.e. $t\in [0,T]$,
    \begin{equation*}
        \begin{aligned}
            \Vert \dot{x}_n(t)\Vert 
            &\leq c(t)\Vert p_{r(t)}(P_n(x_n(t)))\Vert +d(t)+\int_{0}^t \sigma(t,s)ds\\
            &\quad +\int_{0}^t  \sigma(t,s)\Vert p_{r(s)}(P_n(x_n(s)))\Vert ds \\
            &\leq c(t)\Vert x_n(t)\Vert +d(t)+\int_{0}^t\sigma(t,s)ds+\int_{0}^t  \sigma(t,s)\Vert x_n(s)\Vert ds,
        \end{aligned}
    \end{equation*}
where we have used that $x\mapsto p_{r(t)}(x)$ is Lipschitz of constant $1$.
Therefore, by using the inequality $\Vert P_n(x_0)\Vert \leq 
 \Vert x_0\Vert$, we get that for all $t\in [0,T]$
\begin{eqnarray*}
        \Vert x_{n}(t)\Vert 
        &\leq&      \Vert x_{0}\Vert+\int_{0}^{t} d(s)ds+\int_{0}^{t} \int_{0}^{s} \sigma(s,\tau)d\tau ds + \int_{0}^{t} c(s)\Vert x_{n}(s)\Vert ds\\
        &    &      +\int_{0}^{t} \int_{0}^{s} \sigma(s,\tau)\Vert x_{n}(\tau)\Vert d\tau ds,
\end{eqnarray*}
which by virtue of Lemma \ref{Gronwall-2}, implies inequality that $x_n(\cdot)$ satisfies \eqref{cota1}. Finally,  
$$\Vert P_n(x_n(t))\Vert \leq r(t) \textrm{ for all } t\in [0,T].$$
Hence, $p_{r(t)}(P_n(x_n(t)))=P_n(x_n(t))$ for all $t\in [0,T]$, which finishes the proof.
\end{proof}
\begin{remark}\label{obs-Multi}
    It is worth to emphasize that the solution $x_n(\cdot)$ obtained in Theorem \ref{existencia-finito} satisfies for a.e. $t\in [0,T]$
    \begin{equation*}
        \dot{x}_n(t)\in G(t,P_n(x_n(t)))+\int_0^t g(t,s,P_n(x_n(s)))ds,
    \end{equation*}
    where $G(t,x):=F(t,x)\cap (c(t)\Vert x\Vert +d(t))\mathbb{B}$ is the truncation of $F$ to the ball centered at $0$ with radius $(c(t)\Vert x\Vert +d(t))$. Here $c(\cdot)$ and $d(\cdot)$ are the functions defined in \ref{H3F}. Hence, we have reduced a differential inclusion with (possibly) unbounded values to a one satisfying usual linear growth conditions. This fact will be key to the analysis carried out in this article.
\end{remark}

\subsection{Integro-differential inclusions under compactness assumptions}\label{subsection-compactness}
\,\newline In this subsection, we prove a Compactness Principle for integro-differential inclusions, that is, whenever the sequence $(P_n(x_n(t)))_n$ obtained in Theorem \ref{existencia-finito} is relatively compact for all $t\in [0,T]$, then the problem \eqref{diff-inc} admits at least one absolutely continuous solution. This principle is used to show the existence of solutions for the problem \eqref{diff-inc} under a compactness assumption (see Theorem \ref{main}). Subsequently, in Section \ref{Sweeping-sec}, the Principle of Compactness is used to prove the existence of Volterra state-dependent sweeping processes.

\begin{theorem}[Compactness Principle for integro-differential inclusions]\label{main-compactness}  Let assumptions  \ref{HF} and \ref{Hg} hold. Assume that the sequence $(P_n(x_n(t)))_n$ obtained in  Theorem \ref{existencia-finito} is relatively compact for all $t\in [0,T]$. Then, there exists a subsequence $(x_{n_k})_k$ of $(x_n)$ converging strongly pointwisely to a solution $x\in AC([0,T];\H)$ of \eqref{Problema}.  Moreover, 
    \begin{eqnarray*}
        \Vert x(t)\Vert \leq r(t) \textrm{ for all } t\in [0,T] \textrm{ and } \Vert \dot{x}(t)\Vert \leq \psi(t)\textrm{ for a.e. } t\in [0,T],
    \end{eqnarray*}
    where $r$ and $\psi$ are the functions defined in Theorem \ref{existencia-finito}.    
\end{theorem}
\begin{proof}
    We will show the existence of the subsequence via  Lemma \ref{compactness}.
    \begin{claim}{1}{ 
        There exists a subsequence $(x_{n_k})_k$ of $(x_n)_n$ and an absolutely continuous function $x$ such that \ref{comp-i}, \ref{comp-ii}, \ref{comp-iii} and \ref{comp-iv} from Lemma \ref{compactness}  hold with $\psi$ defined as in the statement of the theorem.
    }\end{claim}
    \begin{claimproof}{1}{
        Due to Theorem  \ref{existencia-finito},  $\Vert \dot{x}_n(t)\Vert \leq \psi(t)$ for a.e. $t\in[0,T]$, which shows that \eqref{acotamiento} holds with the function $\psi$ defined as above. Also, $P_n(x_0)\to x_0$ as $n\to +\infty$. Therefore, the claim follows from  Lemma \ref{compactness}.
    }\end{claimproof}\qed\\
    By simplicity we denote $P_k:=P_{n_k}$ and $x_{k}:=x_{n_k}$ for $k\in \mathbb{N}$.
    \begin{claim}{2}{ 
        $P_k(x_k(t)) \rightharpoonup x(t)$ as $k\to +\infty$  for all $t\in [0,T]$.
    }\end{claim}
    \begin{claimproof}{2}{ 
        Since $x_k(t)\rightharpoonup x(t)$ as $k\to+\infty$ for all $t\in [0,T]$, the result follows from assertion \ref{Pro3} of  Lemma  \ref{proyecciones}.
    }\end{claimproof}\qed
    
    \begin{claim}{3}{ 
        $P_{k}(x_{k}(t)) \to x(t)$ as $k\to +\infty$  for all $t\in [0,T]$.
    }\end{claim}    
    \begin{claimproof}{3}{
        The result follows from Claim 2 and the relative compactness of the sequence $(P_{n}(x_{n}(t)))_{n}$ for a.e. $t\in [0,T]$.
    }\end{claimproof} \qed \\
Let us denote 
$$
z_k(t):=\dot{x}_k(t)-\int_0^t g(t,s,P_k(x_k(s)))ds \textrm{ and } z(t):=\dot{x}(t)-\int_0^t g(t,s,x(s))ds.
$$ Then, by virtue of Claim 1 and Claim 3, we have that $z_k \rightharpoonup z$ in $L^1([0,T];\mathcal{H})$. Moreover, if we denote $G(t,x):=F(t,x)\cap (c(t)\Vert x\Vert +d(t))$, the according to Remark \ref{obs-Multi}, we obtain that
$$
z_k(t)\in G(t,P_k(x_k(t))) \quad \textrm{ for a.e. } t\in [0,T].
$$
Hence, summarizing, we have
    \begin{enumerate}[label=(\roman{*})]
        \item For each $x\in \H$, $G(\cdot,x)$ is measurable.
        \item For a.e. $t\in [0,T]$, $G(t,\cdot)$ is upper semicontinuous from $\H$ into $\H_w$.
        \item $z_k\rightharpoonup z$ in $L^1\left([0,T];\H\right)$;
        \item $P_k(x_k(t))\to x(t)$ as $k\to +\infty$ for a.e. $t\in [0,T]$;
        \item For all $k\in \mathbb{N}$, $z_k(t)\in {G}(t,P_k(x_k(t)))$ for a.e. $\in [0,T]$.
    \end{enumerate}
    These conditions and the Convergence Theorem (see \cite[Theorem~3.1.2]{zbMATH00785401}) imply that 
    $$
    z(t)\in G(t,x(t)) \quad \textrm{ for  a.e. } t\in [0,T],
    $$
    that is,  $x\in \operatorname{AC}\left([0,T];\H\right)$ is a solution of \eqref{Problema}, which finishes the proof.
\end{proof}

The following result establishes the existence of solutions for  \eqref{Problema} under the compactness assumption \ref{H4F}. Hence, whenever $g\equiv 0$, we recover classical results of differential inclusions under compactness assumptions (see, e.g., \cite[Chapter~9]{MR1189795}). Furthermore, when the dimension of space $\H$ is finite, this condition is trivially satisfied.
\begin{theorem}\label{main}
    Assume, in addition to the hypotheses of Theorem \ref{existencia-finito}, that  \ref{H4F} holds.  Then, the problem \eqref{Problema} admits at least one absolutely continuous solution $x(\cdot)$. Moreover, 
    \begin{eqnarray*}
        \Vert x(t)\Vert \leq r(t) \textrm{ for all } t\in [0,T] \textrm{ and } \Vert \dot{x}(t)\Vert \leq \psi(t)\textrm{ for a.e. } t\in [0,T],
    \end{eqnarray*}
    where $r$ and $\psi$ are the functions defined in Theorem \ref{existencia-finito}.
\end{theorem}
\begin{proof} 
    We show that the  the sequence  $(P_{n}(x_{n}(t)))_{n}$ is relatively compact for all $t\in [0,T]$. Hence, the result will be a consequence of  Theorem \ref{main-compactness}.  Indeed, let us consider the measurable set-valued mapping
    $$t\mapsto A(t)= \{P_n(x_n(t)) \colon n\in \mathbb{N}\}.$$ 

  \noindent  We proceed to prove that $\beta(A(t))=0$ for all $t\in [0,T]$, where  $\beta$ is the Hausdorff measure of non-compactness. We first observe that, according to Lemma \ref{beta-measurability} and  Theorem \ref{existencia-finito}, the maps 
    $$t\mapsto \beta(\{\dot{x}(t)\colon n\in \mathbb{N}\}) \textrm{ and } 
    t\mapsto \beta\left(\left\{\int_0^t g(t,s,P_n(x_n(s)))ds\colon n\in \mathbb{N}\right\}\right)
    $$
    are integrable. By virtue of  Lemma \ref{Lemma-Constante}, there exists $C>0$ such that    $$\beta(A(t)) \leq C   \beta(\{x_n(t)\colon n\in \mathbb{N}\}) \textrm{ for all } t\in [0,T].$$ 
 Moreover,  by using Proposition \ref{Kura}, the fact that $\{P_{n}(x_{0})\colon n\in \mathbb{N}\}$ is relatively compact and Lemma \ref{2.3}, we obtain that 
    \begin{equation*}
            \begin{aligned}
                \beta(A(t)) 
                &\leq C   \beta(\{x_{n}(t)\colon n\in \mathbb{N} \}) \\
                &   \leq C\beta (\{P_{n}(x_{0})\colon n\in \mathbb{N}\}) + C\beta(\{\int_{0}^{t} \dot{x}_{n}(s)ds\colon n\in \mathbb{N}\})\\
                &   =C\beta(\{\int_{0}^{t} \dot{x}_{n}(s)ds\colon n\in \mathbb{N}\})  \leq C\int_{0}^{t} \beta(\{\dot{x}_{n}(s)\colon n\in \mathbb{N}\})ds.
                % &   \textcolor{blue}{\leq C\int_{0}^{t} \beta(F(s,A(s)))ds+C\int_{0}^{t} \beta \left(\left\{\int_{0}^{s} g(s,\tau,P_n(x_n(\tau)))d\tau \colon n\in \mathbb{N}\right\}\right) ds}\\
                % &   \leq C\int_{0}^{t} \beta(F(s,A(s)))ds + C\int_{0}^{t} \int_{0}^{s}  \beta(g(s,\tau,A(\tau)))d\tau ds\\
                % &   \textcolor{blue}{\leq C\int_{0}^{t} \sqrt{2}k(s)\beta(A(s))ds+C \int_{0}^{t} \int_{0}^{s}  \sqrt{2}\mu_{r(T)}(\tau) \beta(A(\tau))d\tau ds}
            \end{aligned}
        \end{equation*} 
    Now, again by Proposition \ref{Kura} and Lemma \ref{2.3},   we get that for all $s \in [0, T]$ {\small 
    \begin{align*}
    \beta(\{\dot{x}_{n}(s)\colon n\in \mathbb{N}\} &\leq \beta(F(s,A(s)))  + \beta \left(\left\{\int_{0}^{s} g(s,\tau,P_n(x_n(\tau)))d\tau \colon n\in \mathbb{N}\right\}\right) \\
    & \leq \beta(F(s,A(s)))  + \int_{0}^{s}  \beta\left(\left\{     g(s,\tau,P_n(x_n(\tau)))\colon n\in \mathbb{N}\right\}\right) d\tau.
    \end{align*}}
    Furthermore, due to  \eqref{equivalentes} and \ref{H4F}, we obtain that
    \begin{align*}
        \beta(F(s,A(s))) \leq\sqrt{2}k_{r(T)}(s)\beta(A(s)) \text{ for all } s\in [0,T].
    \end{align*}
   On the other hand, using  Lemma \ref{Lemma-measure-Lipschitz} and  \ref{Hg}, we yield   
\begin{align*}
\beta\left(\left\{     g(s,\tau,P_n(x_n(\tau)))\colon n\in \mathbb{N}\right\}\right) \leq  \sqrt{2}\mu_{r(T)}(\tau) \beta(A(\tau)) \textrm{ for all } s\in [0,T].
\end{align*}
     
\noindent Therefore, for all $t\in [0,T]$
        $$
         \beta(A(t))\leq \sqrt{2}C\int_{0}^{t} k_{r(T)}(s)\beta(A(s))ds+\sqrt{2}C \int_{0}^{t} \int_{0}^{s}  \mu_{r(T)}(\tau) \beta(A(\tau))d\tau ds,
        $$
        which, due to Lemma \ref{Gronwall-2},  implies that $\beta(A(t))=0$ for all $t\in [0,T]$. \\
        Therefore, the sequence $(P_{n}(x_{n}(t)))_{n}$ is relatively compact for all $t\in [0,T]$. The proof is finished.
\end{proof}

\subsection{Integro-differential inclusions under monotonicity assumptions}\label{subsection-Lipschitz}
The following result establishes the existence of solutions for   \eqref{Problema} under a monotonicity condition on the set-valued map $F$. Hence, whenever $g\equiv 0$, we recover classical results of differential inclusions under monotonicity conditions (see, e.g., \cite[Chapter~9]{MR1189795}). Furthermore, this result does not require any compactness assumption on the problem data.
\begin{theorem}\label{main-Lipschitz}
    Assume, in addition to the hypotheses of Theorem \ref{existencia-finito}, that  \ref{H5F} holds.  Then, the problem \eqref{Problema} has a unique absolutely continuous solution $x$. Moreover, 
    \begin{eqnarray*}
        \Vert x(t)\Vert \leq r(t) \textrm{ for all } t\in [0,T] \textrm{ and } \Vert \dot{x}(t)\Vert \leq \psi(t)\textrm{ for a.e. } t\in [0,T],
    \end{eqnarray*}
    where $r$ and $\psi$ are the functions defined in Theorem \ref{existencia-finito}.
\end{theorem}

\begin{proof} We will prove that  $(P_n x_n(t))_n$ is a Cauchy sequence for all $t\in [0,T]$. Indeed, fix $n, m\in \mathbb{N}$ with $n\geq m+1$. Let us consider the absolutely continuous function:
    \begin{eqnarray*}
        \Theta(t):=\frac{1}{2}\Vert P_n x_n(t)-P_m x_m(t)\Vert^2.
    \end{eqnarray*}       
    Then, for a.e. $t\in [0,T]$, 
    \begin{equation*}
        \begin{aligned}
            \dot{\Theta}(t)&=\langle P_n x_n(t)-P_m x_m(t), \dot{x}_n(t)-\dot{x}_m(t)\rangle + \sum_{k=m+1}^n \langle x_n(t),e_k\rangle \langle \dot{x}_m(t),e_k\rangle.
        \end{aligned}
    \end{equation*}

    Set $\delta_{n,m}(t):=\sum_{k=m+1}^n \langle x_n(t),e_k\rangle \langle \dot{x}_m(t),e_k\rangle$. Then,  for a.e. $t\in [0,T]$
    \begin{equation*}
        \begin{aligned}
            \vert\delta_{n,m}(t)\vert
            &\leq \left(\sum_{k=m+1}^n \langle x_n(t),e_k\rangle^ 2\right)^{1/2} \cdot \left(\sum_{k=m+1}^n \langle \dot{x}_m(t),e_k\rangle^2\right)^{1/2}\\
            &\leq \Vert x_n(t)\Vert \cdot \Vert \dot{x}_m(t)\Vert\leq r(t)\cdot \psi(t),
        \end{aligned}
    \end{equation*}    
    where $r$ and $\psi$ are the functions defined in Theorem \ref{existencia-finito}.     Therefore, $\delta_{n,m}$ is uniformly bounded in $L^1([0,T];\mathbb{R})$ and $\delta_{n,m}(t)$ converges to $0$, as $n, m \to +\infty$,  for all $t\in [0,T]$.

    Besides, by using \ref{H5F}, we obtain that for a.e. $t\in [0,T]$, one has
    \begin{equation*}
        \begin{aligned}
            &\dot{\Theta}(t)
               \leq \tilde{k}_{r(T)}(t)\Vert P_{n} x_{n}(t)-P_{m} x_{m}(t)\Vert^{2}\\
            &+\langle P_n x_n(t)-P_m x_m(t), \int_{0}^t  [g(t,s,P_{n}x_{n}(s)) - g(t,s,P_{m}x_{m}(s))] ds \rangle + \vert \delta_{n,m}(t) \vert \\
            &   \leq \tilde{k}_{r(T)}(t)\Vert P_n x_n(t)-P_m x_m(t)\Vert^2\\
            &+\Vert P_{n} x_{n}(t)-P_{m} x_{m}(t)\Vert  \int_{0}^{t} \mu_{r(T)}(s) \Vert P_{n} x_{n}(s)-P_{m} x_{m}(s)\Vert ds + \vert \delta_{n,m}(t)\vert. 
        \end{aligned}
    \end{equation*}
Hence, for all $t\in [0,T]$ 
$$
\dot{\Theta}(t)\leq 2\tilde{k}_{r(T)}(t)\Theta(t)+\sqrt{2\Theta(t)}\int_{0}^t \mu_{r(T)}(s)\sqrt{2\Theta(s)}ds+\vert \delta_{n,m}(t)\vert. 
$$
Therefore, by arguments similar to those given in the proof of Lemma \ref{Gronwall-2}, we get that, for all $t\in [0,T]$, one has 
{\small
    \begin{equation*}
        \begin{aligned}
            \Theta(t)\leq \frac{1}{2}\Vert P_n x_0-P_m x_0\Vert^2 \exp\left(\int_0^t \upsilon(s)ds\right)+\int_{0}^t \vert \delta_{n,m}(s)\vert \exp\left(\int_s^t \upsilon(\tau)d\tau\right)ds,
        \end{aligned}
    \end{equation*}}
    where $\upsilon(t):= 2\tilde{k}_{r(T)}(t)+2\int_{0}^{t}\mu_{r(T)}(s)ds$ for $t\in [0,T]$.    Then, by taking the limit in the above inequality, we obtain that $(P_nx_n(t))_n$ is a Cauchy sequence for all $t\in [0,T]$. Hence,   for some $x\colon [0,T]\to \H$, the following convergence holds
    \begin{equation*}
        P_n x_n(t)\to x(t) \textrm{ for all } t\in [0,T].
    \end{equation*}
     To prove that $x(\cdot)$ is indeed a solution of \eqref{Problema}, we can proceed similarly to the proof of  Theorem \ref{main}.  Finally, to prove the uniqueness, let $x_1$ and $x_2$ be two solutions of \eqref{Problema} and define $\vartheta(t):=\frac{1}{2}\Vert x_1(t)-x_2(t)\Vert^2$. Hence, for a.e. $t\in [0,T]$,   {\small
     \begin{equation*}
        \begin{aligned}
         \hspace{-0.8mm}   \dot{\vartheta}(t)
               &\leq \tilde{k}_{r(T)}(t)\Vert x_1(t)-x_2(t)\Vert^{2}
            +\langle x_1(t)-x_2(t), \int_{0}^t [g(t,s,x_1(s)) - g(t,s,x_2(s))] ds \rangle \\
            & \leq \tilde{k}_{r(T)}(t)\Vert x_1(t)-x_2(t)\Vert^2
              +\Vert x_1(t)-x_2(t)\Vert  \int_{0}^{t} \mu_{r(T)}(s) \Vert x_1(s)-x_2(s)\Vert ds,
        \end{aligned}
    \end{equation*}}
 which, by arguments similar to those given in the proof of Lemma \ref{Gronwall-2}, implies that $\vartheta(t)=0$ for all $t\in [0,T]$.  
\end{proof}

\section{Volterra State-Dependent Sweeping Processes}\label{Sweeping-sec}

In this section, by using the Approximation Principle (Theorem \ref{existencia-finito}) and the  Compactness Principle (Theorem \ref{main-compactness}), we obtain the existence of solutions for the state-dependent sweeping process with integral perturbation:
\begin{equation}\label{Sweeping-Dif}
\left\{
\begin{aligned}
\dot{x}(t)&\in -N(C(t,x(t));x(t))+F(t,x(t))+\int_{0}^t g(t,s,x(s))ds  \textrm{ a.e }  t\in [0,T],\\
x(0)&=x_0,
\end{aligned}
\right.
\end{equation}
where $x_0\in C(0,x_0)$ and $C\colon [0,T]\times \H \rightrightarrows \H$ is set-valued with nonempty and closed values satisfying the following assumptions:
\begin{enumerate}
\item[\namedlabel{Hx1}{$(\mathcal{H}_1)$}] There exist $\zeta\in \operatorname{AC}\left([0,T];\mathbb{R}_+\right)$ and $L\in[0,1[$ such that for all $s,t\in [0,T]$ and all $x,y,z\in \H$
\begin{equation*}
|d(z,C(t,x))-d(z,C(s,y))|\leq |\zeta(t)-\zeta(s)|+L\Vert x-y\Vert.
\end{equation*}

\item[\namedlabel{Hx2}{$(\mathcal{H}_2)$}] The family $\{C(t,x)\colon (t,x)\in [0,T]\times \H\}$ is equi-uniformly subsmooth.

\item[\namedlabel{Hx3}{$(\mathcal{H}_3)$}] For every $t\in [0,T]$, every $r>0$ and every bounded set $A\subset \H$ the set $C(t,A)\cap r\mathbb{B}$ is relatively compact.

\item[\namedlabel{Hx4}{$(\mathcal{H}_4)$}] There exist $\alpha_0\in ]0,1]$ and $\rho\in ]0,+\infty]$ such that for every $x\in \H$
\begin{equation*}
\begin{aligned}
  0<\alpha_0&\leq \inf_{z\in U_{\rho}\left(C(t,x)\right)}d\left(0,\partial d(\cdot,C(t,x))(z)\right) & \textrm{ a.e. } t\in [0,T],
  \end{aligned}
\end{equation*}
where $U_{\rho}\left(C(t,x)\right)=\left\{ z\in \H \colon 0<d(z,C(t,x))<\rho \right\}$.
\end{enumerate}

\begin{remark}\label{remarque}
According to  \cite[Proposition 3.9]{JV-alpha}, hypothesis \ref{Hx2} implies that for every $\alpha_0\in ]\sqrt{L},1]$ there exists $\rho>0$ such that \ref{Hx4} holds. 
\end{remark}

The next theorem is the main result of this section. Whenever $g\equiv 0$, it extends previous results on state-dependent sweeping processes (see, e.g., \cite{JV-alpha,JV-regular,JV-Galerkin}). Moreover, it complements the results from \cite{bouach2021nonconvex}, where the authors prove the existence of solutions for integrally perturbed sweeping processes with uniformly prox-regular sets. The next result is the first existence result for integrally perturbed state-dependent sweeping processes. We emphasize here that no compactness assumptions on the set-valued map $F$ are assumed. Moreover, the existence is obtained within the framework of equi-uniformly subsmooth sets, which strictly contain the class of prox-regular sets.

\begin{theorem}\label{main-sweeping}
Assume, in addition to  \ref{HF} and \ref{Hg},   that  \ref{Hx1}, \ref{Hx2} and \ref{Hx3} hold.  Fix $\alpha_0\in ]\sqrt{L},1]$ and  $\rho>0$ such that \ref{Hx4} holds (see Remark \ref{remarque}). Let us consider $m\in L^1([0,T];\mathbb{R}_+)$ be  the unique solution for the following first-order integral equation: {\small 
\begin{equation}\label{Def-m}
m(t):=\frac{\vert \dot{\zeta}(t)\vert}{\alpha_0^2 -L} +\frac{(1+L)}{\alpha_0^2-L}\psi(t)+\frac{L}{\alpha_0^2-L}\eta(t)\omega(t),
\end{equation}}
where $r$, $\eta$ and $\varepsilon$ are defined in Theorem \ref{existencia-finito} and set  
\begin{equation}\label{def_r}
\omega(t):=\int_0^t m(s)\exp\left(\int_s^t \eta(\tau)d\tau\right)ds.
\end{equation}
Then, the problem \eqref{Sweeping-Dif} admits at least one solution $x\in \operatorname{AC}([0,T];\H)$ such that $\Vert x(t)\Vert \leq r(t)+\omega(t)$ for all $t\in [0,T]$ and 
\begin{equation}\label{der-5.4}
\Vert \dot{x}(t)\Vert \leq \psi(t)+\eta(t)\omega(t)+m(t) \textrm{ for a.e. } t\in [0,T].
\end{equation}
\end{theorem}
\begin{proof} 
We will prove the theorem under the additional assumption:
\begin{equation}\label{condition-rho}
\int_{0}^T \left(\vert \dot{\zeta}(s)\vert +(1+L)(\psi(s)+\eta(s)\omega(s)+m(s)) \right)ds <\rho,
\end{equation}
where $\rho>0$ is defined as in the statement of the theorem and $\psi$ is defined in \eqref{der-5.4}. The general case, without the above assumption on the length of $T$, can be obtained in a similar way as \cite[Step~2]{JV-Galerkin}. \newline
\noindent  Let $H\colon [0,T]\times \H \to \H$ be the set-valued map defined as
$$
H(t,x):=-m(t)\partial d_{C(t,x)}(x)+F(t,x)%\textcolor{blue}{\cap (c(t)\Vert x\Vert +d(t))\mathbb{B}} \textrm{ for } (t,x)\in [0,T]\times \H,
$$
where $m$ is defined in \eqref{Def-m}. 
We will show, by using the results from Section \ref{partialGalerkin}, that the following differential inclusion has at least one solution:
\begin{equation*}
\left\{
\begin{aligned}
\dot{x}(t)&\in H(t,x(t))+\int_0^t g(t,s,x(s))ds  \textrm{ a.e. } t\in [0,T],\\
x(0)&=x_0.
\end{aligned}
\right.
\end{equation*} 
Moreover, any solution to the above problem solves \eqref{Sweeping-Dif}.

The proof is divided into several claims. \newline \noindent
\emph{Claim 1}: $H$ satisfies the assumptions of Theorem \ref{existencia-finito}, that is, 
\begin{itemize}
\item[(i)] The map $t\rightrightarrows \gph H(t,\cdot)$ is measurable.
\item[(ii)] For a.e. $t\in [0,T]$, $\gph H(t,\cdot)$ is closed on $\H \times \H_w$.
\item[(iii)] For all $x\in \H$ and a.e. $t\in [0,T]$
$$
 d\left(0,H(t,x)\right)  \leq  m(t)+c(t)\Vert x\Vert +d(t),
$$
\end{itemize}
where $c, d$ and $m$ are defined by \ref{H3F} and \eqref{Def-m}, respectively.

\emph{Proof of Claim 1:}
 First, let us check (ii). Indeed, consider $t\in [0,T]$, on the one hand, the mapping $x \mapsto \partial d_{C(t,x)}(x) $ has a closed graph in $\H \times \H_w$ and compact values due to Lemma \ref{Clarke-measurable}; on the other hand, the set-valued mapping $F(t, \cdot)$ has a closed graph in $\H \times \H_w$ due to our assumption \ref{H2F}. Then, $\gph H(t,\cdot)$ is closed on $\H \times \H_w$.

Now, let us check (i). Let us define the auxiliary set-valued mappings  
$$J^1(t, x) = -m(t)\partial d_{C(t,x)}(x)\times \H  \textrm{ and } J^2(t,x) = \H\times F(t,x).$$
It is easy to see that $t \rightrightarrows \gph J^2(t, \cdot) $ is measurable. Furthermore, by  Lemma  \ref{Clarke-measurable},  the set-valued mapping $t \rightrightarrows \gph J^2(t, \cdot) $  is measurable. Now, we notice that 
\begin{align*}
\gph H(t,\cdot) = \varphi \left (  \gph J^1(t,\cdot) \cap \gph J^2(t,\cdot)  \right), 
\end{align*}
where $\varphi$ is defined as $\varphi(x,u,v):= (x,u+v)$. Then, by virtue of \cite[Theorems 8.2.4~and~8.2.8]{Aubin_Frankowska_2009_book}, we get the measurability of $t\rightrightarrows \gph H(t,\cdot)$.  Finally, we observe that
\begin{equation*}
\begin{aligned}
 d\left(0,H(t,x)\right) & \leq m(t)+ d\left(0,F(t,x)\right)\\
& \leq  m(t) + c(t)\Vert x\Vert+d(t),
\end{aligned}
\end{equation*}
which proves (iii).
  \hfill $\square$

 \noindent For each $n\in \mathbb{N}$, let us consider the following integro-differential inclusion:
\begin{equation}\label{eq.22}\left\{
\begin{aligned}
\dot{x}(t)&\in H(t,P_n(x(t)))+\int_0^t g(t,s,P_n(x(s)))ds  \textrm{ a.e. } t\in [0,T],\\
x(0)&=P_n(x_0),
\end{aligned}
\right.
\end{equation} 
where $(P_n)_n$ is an orthonormal basis of $\H$. By virtue of Theorem \ref{existencia-finito}, the above differential inclusion has at least one solution $x_n\in \operatorname{AC}([0,T];\H)$. Moreover, 
\begin{equation*}
\Vert x_n(t)\Vert \leq {r}(t)+\omega(t) \textrm{ for all } t\in [0,T],
\end{equation*}
and for a.e. $t\in [0,T]$, one has
\begin{equation*}
\Vert \dot{x}_n(t)\Vert \leq \psi(t)+\eta(t)\omega(t)+m(t),
\end{equation*}
where  $\omega$ is defined in \eqref{def_r}.
To simplify the notation, we write $$\Gamma_n(t):=\partial d_{C(t,P_n(x_n(t)))}(P_n(x_n(t))).$$
Without loss of generality (see 
  inequality \eqref{condition-rho}), we assume that that $n\in \mathbb{N}$ is large enough so that
\begin{equation}\label{condition-rho2}
\int_{0}^T \left(\vert \dot{\zeta}(s)\vert +(1+L)(\psi(s)+\eta(s)\omega(s)+m(s)) \right)ds+(1+L)\Vert x_0-P_n(x_0)\Vert <\rho.
\end{equation}
\noindent We observe that there exist $f_n(t)\in F(t,P_n(x_n(t)))$ and $d_n(t)\in \Gamma_n(t)$ such that
$$
\dot{x}_n(t)=-m(t)d_n(t)+f_n(t)+\int_{0}^t g(t,s,P_n(x_n(s)))ds \quad \textrm{ a.e. } t\in [0,T].
$$
\noindent Define $\varphi_n(t):=d_{C(t,P_n(x_n(t)))}(P_n(x_n(t)))$ for $t\in [0,T]$. \\ 
\noindent \emph{Claim 2:} The distance function satisfies the following inequality: $$\varphi_n(t)\leq (1+L)\Vert x_0-P_n(x_0)\Vert \textrm{ for all } t\in [0,T].$$
 \begin{claimproof}{2}{ The idea is to estimate the derivative of the distance function $\varphi_n(t)$. To do that, we proceed to show first that $\varphi_n(t)<\rho$ for all $t\in [0,T]$. Indeed, let $t\in [0,T]$ where $\dot{x}_n(t)$ exists. Then, due to \cite[Lemma~4.4]{JV-Galerkin} and \eqref{Def-m}, 
\begin{equation*}
\begin{aligned}
\dot{\varphi}_n(t)&\leq \vert \dot{\zeta}(t)\vert +L\Vert P_n(\dot{x}_n(t))\Vert +\max_{y^{\ast}\in \Gamma_n(t)}\langle y^{\ast},P_n(\dot{x}_n(t))\rangle\\
&\leq \vert \dot{\zeta}(t)\vert +(1+L)\Vert \dot{x}_n(t)\Vert \\
&\leq  \vert \dot{\zeta}(t)\vert +(1+L)(\psi(t)+\eta(t)\omega(t)+m(t)),
\end{aligned}
\end{equation*}
which, by \eqref{condition-rho2}, implies that $\varphi_n(t)<\rho$ for all $t\in [0,T]$. \newline \noindent 
Let $t\in \Omega_n:=\{t\in [0,T]\colon P_n(x_n(t))\notin C(t,P_n(x_n(t)))\}$, where $\dot{x}_n(t)$ exists. Then, due to \cite[Lemma~4.4]{JV-Galerkin}, one has
\begin{equation*}
\begin{aligned}
\dot{\varphi}_n(t)&\leq \vert \dot{\zeta}(t)\vert +L\Vert P_n(\dot{x}_n(t))\Vert+\min_{y^{\ast}\in \Gamma_n(t)}\langle y^{\ast},P_n(\dot{x}_n(t))\rangle\\
&\leq \vert \dot{\zeta}(t)\vert +L(\psi(t)+\eta(t)\omega(t)+m(t))+\langle d_n(t),P_n(\dot{x}_n(t))\rangle\\
&=\vert \dot{\zeta}(t)\vert +L(\psi(t)+\eta(t)\omega(t)+m(t))+m(t)\langle d_n(t),-P_n(d_n(t))\rangle\\
&+\langle d_n(t),P_n(f_n(t)+\int_{0}^t g(t,s,P_n(x_n(s)))ds)\rangle\\
&\leq \vert \dot{\zeta}(t)\vert +L(\psi(t)+\eta(t)\omega(t)+m(t))+\Vert f_n(t)\Vert \\
&+\int_{0}^t \Vert g(t,s,P_n(x_n(s)))\Vert ds+m(t)\langle d_n(t),-P_n(d_n(t))\rangle \\
&\leq \vert \dot{\zeta}(t)\vert +L(\psi(t)+\eta(t)\omega(t)+m(t))+c(t)r(t)+d(t)+ \int_0^t \sigma(t,s)ds\\
&+\int_0^t \sigma(t,s)r(s)ds+m(t)\langle d_n(t),-P_n(d_n(t))\rangle\\
&\leq  \vert \dot{\zeta}(t)\vert +L(\psi(t)+\eta(t)\omega(t)+m(t))+\psi(t)+m(t)\langle d_n(t),-P_n(d_n(t))\rangle,
\end{aligned}
\end{equation*}
where we have used that \ref{HF}, \ref{Hg} and that $d_n(t)\in \Gamma_n(t)$ a.e. $t\in [0,T]$. Moreover, due to \ref{Hx4}, 
\begin{equation*}
\begin{aligned}
\langle d_n(t),-P_n(d_n(t))\rangle&=\langle d_n(t),d_n(t)-P_n(d_n(t))\rangle +\langle d_n(t),-d_n(t)\rangle\\
&\leq \langle d_n(t),d_n(t)-P_n(d_n(t))\rangle-\alpha_0^2\\
&=-\alpha_0^2. 
\end{aligned}
\end{equation*}
Thus, by using the above inequalities and the definition of $m(\cdot)$, we obtain that
\begin{equation*}
\begin{aligned}
\dot{\varphi}_n(t)\leq m(t)(\alpha_0^2 +\langle d_n(t),-P_n(d_n(t))\rangle)\leq 0,
\end{aligned}
\end{equation*}
which, by virtue of \ref{Hx1},  implies that 
\begin{equation*}
\begin{aligned}
    \varphi_n(t)&\leq \varphi_n(0)=d_{C(0,P_n(x_0))}(P_n(x_0))\\&=d_{C(0,P_n(x_0))}(P_n(x_0))-d_{C(0,x_0)}(x_0)\\
    &\leq (1+L)\Vert x_0-P_n(x_0)\Vert.
\end{aligned}
\end{equation*}}\end{claimproof} 
\begin{claim}{3}{  For all $t\in [0,T]$, $\lim_{n\to +\infty}\varphi_n(t)=0$.  }\end{claim}
 \begin{claimproof}{3}{ It follows directly from Claim 2 and  Lemma \ref{proyecciones}.   }\end{claimproof} 
\begin{claim}{4}{  The sequence $(P_n(x_n(t)))_n$ is relatively compact for all $t\in [0,T]$. }\end{claim}
 \begin{claimproof}{4}{ Let $\gamma=\alpha$ or $\gamma=\beta$ be either the Kuratowski or the Hausdorff measure of noncompactness. Fix $t\in [0,T]$  and let 
$$
s_n(t)\in \operatorname{Proj}_{C(t,P_n(x_n(t)))}(P_n(x_n(t))).
$$
Then, $s_n(t)\in (\rho+r(t)+\omega(t))\mathbb{B}$ and, due to Claim 3 and \ref{Hx3}, 
\begin{equation*} 
\begin{aligned}
\gamma(\{P_n(x_n(t))\colon n\in \mathbb{N}\})&=\gamma(\{s_n(t)\colon n\in \mathbb{N}\})\\
&\leq \gamma\left(C(t,(r(t)+\omega(t))\mathbb{B})\cap (\rho +r(t)+\omega(t))\mathbb{B}\right)=0, 
\end{aligned}
\end{equation*}
which proves the claim. }\end{claimproof} \\  \noindent Hence, we have verified all hypotheses of  Theorem \ref{main-compactness}, and there exists a subsequence $(x_{n_k})_k$ of $(x_n)$ converging strongly pointwisely to an absolutely continuous solution $x(\cdot)$ of \eqref{eq.22}. 
 \begin{claim}{5}{ For all $t\in [0,T]$, $x(t)\in C(t,x(t))$. }\end{claim}
 \begin{claimproof}{5}{ Fix $t\in [0,T]$. Then, as in the proof of Theorem \ref{main-compactness}, $P_k(x_{n_k}(t))\to x(t)$ for some subsequence $(x_{n_k})_k$ of $(x_{n})_n$. Hence, due to \ref{Hx1} and Claim 3, one has
\begin{equation*}
\begin{aligned}
d_{C(t,x(t))}(x(t))&=\limsup_{k\to +\infty}\left( d_{C(t,x(t))}(x(t))-\varphi_{n_k}(t)+\varphi_{n_k}(t)\right)\\
&\leq \limsup_{k\to +\infty}((1+L)\Vert x(t)-P_{n_k}(x_{n_k}(t))\Vert +\varphi_{n_k}(t))=0,
\end{aligned}
\end{equation*}
as claimed.  }
\end{claimproof} \\ 
Finally, by virtue of formula \eqref{eq.13} and Claim 5, $x(\cdot)$ is also a solution \eqref{Sweeping-Dif}. 
\end{proof}

%\begin{remark}\label{Reduction}
%The proof of Theorem \ref{main-sweeping} reveals that the obtained solution $x$ satisfies the following differential inclusion:  
%\begin{equation*}{\small
%\left\{
%\begin{aligned}
%&\dot{x}(t)\in -m(t)\partial  d_{C(t,x(t))}(x(t))+F(t,x(t))+\int_0^t g(t,s,x(s))ds & \textrm{ a.e. } t\in [0,T];\\
%&x(t)\in C(t,x(t)) & \textrm{ for all } t\in [0,T];\\
%&x(0)=x_0\in C(0,x_0).
%\end{aligned}
%\right.}
%\end{equation*}
%Moreover, according to the formula \eqref{eq.13}, any solution of the above differential inclusion is a solution of \eqref{Sweeping-Dif}.
%\end{remark}

\paragraph{Existence of solutions for the integro-differential sweeping process} We end this section with an existence result for Volterra sweeping process. 
\begin{equation}\label{Sweeping-Dif-Class}
\left\{
\begin{aligned}
\dot{x}(t)&\in -N(C(t);x(t))+F(t,x(t))+\int_{0}^t g(t,s,x(s))ds  \textrm{ a.e }  t\in [0,T],\\
x(0)&=x_0,
\end{aligned}
\right.
\end{equation}
where $C\colon [0,T] \rightrightarrows \H$ is set-valued with nonempty and closed values satisfying the following assumptions:
\begin{enumerate}
\item[\namedlabel{HAC}{$(\mathcal{H}_5)$}] There exists $\zeta\in \operatorname{AC}\left([0,T];\mathbb{R}_+\right)$ such that for all $s,t\in [0,T]$
\begin{equation*}
\sup_{x\in \H}|d(x,C(t))-d(x,C(s))|\leq |\zeta(t)-\zeta(s)|.
\end{equation*}

\item[\namedlabel{Halpha}{$(\mathcal{H}_{6})$}] There exist two constants $\alpha_0\in ]0,1]$ and $\rho\in ]0,+\infty]$ such that
\begin{equation*}
\begin{aligned}
  0<\alpha_0&\leq \inf_{x\in U_{\rho}\left(C(t)\right)}d\left(0,\partial d(x,C(t))\right) & \textrm{ a.e. } t\in [0,T],
  \end{aligned}
\end{equation*}
where $U_{\rho}\left(C(t)\right)=\left\{ x\in \H \colon 0<d(x,C(t))<\rho \right\}$ for all $t\in [0,T]$.

\item[\namedlabel{Hcomp}{$(\mathcal{H}_{7})$}] For all $t\in [0,T]$ the set $C(t)$ is ball-compact, that is, for every $r>0$ the set $C(t)\cap r\mathbb{B}$ is compact in $\H$.
\end{enumerate}
The following result can be proved similarly to Theorem \ref{main-sweeping}. When $\H$ is a finite-dimensional space, it extends the result from \cite{bouach2021nonconvex} to positively $\alpha$-far sets, a class which strictly includes prox-regular sets. Moreover, it extends the results from \cite{JV-alpha} by incorporating an integral perturbation.
\begin{theorem} 
Assume, in addition to \ref{HF} and \ref{Hg}, that  \ref{HAC}, \ref{Halpha} and \ref{Hcomp} hold. Then, the differential inclusion  \eqref{Sweeping-Dif-Class} admits at least one solution $x\in \operatorname{AC}([0,T];\H)$ such that $\Vert x(t)\Vert \leq r(t)+\omega(t)$ for all $t\in [0,T]$ and 
$$
\Vert \dot{x}(t)\Vert \leq \psi(t)+\eta(t)\omega(t)+m(t) \textrm{ for a.e. } t\in [0,T],
$$
where $m(t):=\dfrac{\vert \dot{\zeta}(t)\vert +\psi(t)}{\alpha_0^2}$ and $\omega(t):=\int_0^t m(s)\exp\left(\int_s^t \eta(\tau)d\tau\right)ds$.
\end{theorem}

\section{Reduction Technique for Volterra State-Dependent Sweeping Processes}\label{Reduction-sp}
In this section, we provide a reduction technique for the Volterra state-dependent sweeping process \eqref{Sweeping-Dif}.  This result will be used in Section \ref{Control-problem} to prove the existence of solutions of optimal controls. 

In order to obtain the reduction result of this section, we strengthen $(\mathcal{H}_3^F)$ to the condition:  \vspace{0.5mm}
\begin{enumerate}
    \item[\namedlabel{H3Fs}{$(\mathcal{H}_3^{F_s})$}] There exist two nonnegative integrable functions $c,d$ such that for all $x\in \H$ and a.e. $t\in [0,T]$,  $\|F(t,x)\| \leq  c(t)\Vert x\Vert +d(t)$. 
    \end{enumerate} 
The proof of the next result is based on ideas from \cite{Thibault2003}. As in the proof of Theorem \ref{main-sweeping}, we will show that the differential inclusion \eqref{Sweeping-Dif} is related to the following unconstrained differential inclusion:
\begin{equation}\label{reducido-dif-eq} 
\left\{
\begin{aligned}
&\dot{x}(t)\in -\nu(t)\partial  d_{C(t,x(t))}(x(t))+F(t,x(t))+\int_0^t g(t,s,x(s))ds & \textrm{a.e. } t\in I;\\
&x(0)=x_0\in C(0,x_0),
\end{aligned}
\right.
\end{equation}
where $\nu\colon [0,T]\to \mathbb{R}_+$ is defined as
$$
\nu(t):=\frac{\vert \dot{\zeta}(t)\vert}{1-L}+\frac{1+L}{1-L}(\eta(t)R(t)+\varepsilon(t)),  
$$
$\eta(\cdot)$, $\varepsilon(\cdot)$ are defined in Theorem \ref{existencia-finito} and 
\begin{equation*}
\begin{aligned}
R(t)&:=\Vert x_0\Vert \exp\left(\frac{2}{1-L}\int_0^t \eta(s)ds\right)\\
    &+\frac{1}{1-L}\int_0^t  (\vert \dot{\zeta}(s)\vert+2\varepsilon(s))\exp\left(\frac{2}{1-L}\int_s^t \eta(\tau)d\tau\right)ds.
\end{aligned}
\end{equation*}

The result is the following:
\begin{proposition}\label{reduction-thm}
    Suppose, in addition to \ref{H1F}, \ref{H2F}, \ref{H3Fs}, that \ref{Hx1} and \ref{Hx2} hold. Then, any solution $x(\cdot)$ of \eqref{Sweeping-Dif} is a solution of unconstrained differential inclusion \eqref{reducido-dif-eq}. Reciprocally, any solution $x(\cdot)$ of \eqref{reducido-dif-eq} satisfying $x(t)\in C(t,x(t))$ for all $t\in [0,T]$ is a solution of \eqref{Sweeping-Dif}. 
\end{proposition}
\begin{proof} 
The second assertion follows directly from \eqref{eq.13}. To prove the first one, let $x$ be a solution of \eqref{Sweeping-Dif} and let $f(\cdot)$ be a measurable selection of $F(t,x(t))$ such that 
$$
\dot{x}(t)\in -N(C(t,x(t));x(t))+f(t)+\int_0^t g(t,s,x(s))ds \textrm{ for a.e. } t\in [0,T].
$$
Then, for a.e. $t\in [0,T]$, one has
$$
\dot{x}(t)\in -N(C(t,x(t));x(t))+z(t),
$$
where $z(t):=f(t)+\int_{0}^t g(t,s,x(s))ds$. Define, for all $t\in [0,T]$, 
$$
v(t):=\int_0^t z(s)ds, D(t):=C(t,x(t))-v(t) \textrm{ and } y(t):=x(t)-v(t).
$$
Hence, for a.e. $t\in [0,T]$, one has
\begin{equation*}
\dot{y}(t)\in -N(D(t);y(t)). 
\end{equation*}
Moreover,  according to \ref{Hx1} and \ref{Hx2}, it is routine to verify that $D(t)$ is normally regular for all $t\in [0,T]$ and that for all $x\in \mathcal{H}$ and $s<t$, one has  {\small 
\begin{equation*}
\begin{aligned}
    \vert d_{D(t)}(x)-d_{D(s)}(x)\vert &\leq \vert \zeta(t)-\zeta(s)\vert+L\Vert x(t)-x(s)\Vert +\Vert v(t)-v(s)\Vert \\
    &\leq \vert \zeta(t)-\zeta(s)\vert+L\Vert x(t)-x(s)\Vert +\int_s^t\Vert f(s)\Vert ds\\
    &+\int_s^t \int_0^r \Vert g(r,\tau,x(\tau))\Vert d\tau dr \\
    &\leq \int_s^t \vert \dot{\zeta}(\tau)\vert d\tau+L\int_s^t \Vert \dot{x}(\tau)\Vert d\tau +\int_s^t (c(s)\Vert x(s)\Vert +d(s))ds \\
    &+\int_s^t \int_0^r \sigma(r,\tau)(1+\Vert x(\tau)\Vert )d\tau dr,
    \end{aligned}
\end{equation*}}
where we have used \ref{H3Fs} and \ref{Hg}$(b)$.  Therefore, by virtue of \cite[Proposition~2.1]{Thibault2003}, we obtain that
\begin{equation}\label{sweeping-y}
\dot{y}(t)\in -\theta(t)\partial d_{D(t)}(y(t)) \textrm{ a.e. } t\in [0,T],
\end{equation}
where 
$$
\theta(t):=\vert \dot{\zeta}(t)\vert+L\Vert \dot{x}(t)\Vert +c(t)\Vert x(t)\Vert +d(t)+\int_0^t \sigma(t,s)(1+\Vert x(s)\Vert )ds.
$$
The above inclusion implies that, for a.e. $t\in [0,T]$, one has
\begin{equation*}
\begin{aligned}
\Vert \dot{x}(t)\Vert &\leq \Vert \dot{y}(t)\Vert +\Vert\dot{v}(t)\Vert \\ 
&\leq \theta(t)+\Vert z(t)\Vert\\
&\leq \theta(t)+\Vert f(t)\Vert +\int_0^t \Vert g(t,s,x(s))\Vert ds\\
&\leq \theta(t)+c(t)\Vert x(t)\Vert +d(t) +\int_0^t \sigma(t,s)(1+\Vert x(s)\Vert )ds,
\end{aligned}
\end{equation*}
which gives that for a.e. $t\in [0,T]$, one has
\begin{equation}\label{a-priori-c}
\Vert \dot{x}(t)\Vert \leq \frac{\vert \dot{\zeta}(t)\vert }{1-L}+2\frac{c(t)}{1-L}\Vert x(t)\Vert +2\frac{d(t)}{1-L}+2\int_0^t \frac{\sigma(t,s)}{1-L}(1+\Vert x(s)\Vert )ds.
\end{equation}
From  Lemma \ref{Gronwall-2}, we obtain that, for all $t\in [0,T]$, one has
\begin{equation*}
\begin{aligned}
    \Vert x(t)\Vert &\leq R(t):=\Vert x_0\Vert \exp\left(\frac{2}{1-L}\int_0^t \eta(s)ds\right)\\
    &+\frac{1}{1-L}\int_0^t  (\vert \dot{\zeta}(s)\vert+2\varepsilon(s))\exp\left(\frac{2}{1-L}\int_s^t \eta(\tau)d\tau\right)ds.
    \end{aligned}
\end{equation*}
 Moreover, from inequality \eqref{a-priori-c}, we obtain that for a.e. $t\in [0,T]$, one has  
\begin{equation*}
\begin{aligned}
\Vert \dot{x}(t)\Vert &\leq  
 \frac{\vert \dot{\zeta}(t)\vert }{1-L}+2\frac{c(t)}{1-L}R(t) +2\frac{d(t)}{1-L}+2\int_0^t \frac{\sigma(t,s)}{1-L}(1+R(s) )ds\\
 &\leq \varrho(t):= \frac{\vert \dot{\zeta}(t)\vert }{1-L}+\frac{2}{1-L}\left(\eta(t)R(t)+\varepsilon(t)\right).
 \end{aligned}
 \end{equation*} 
On the other hand, due to \ref{H3Fs} and \ref{Hg}$(b)$, we get that for a.e. $t\in [0,T]$, one has 
\begin{equation*}
    \begin{aligned}
      &  \Vert \dot{x}(t)-z(t)\Vert=\Vert \dot{y}(t)\Vert \\
        &\leq \theta(t)=\vert \dot{\zeta}(t)\vert+L\Vert \dot{x}(t)\Vert +c(t)\Vert x(t)\Vert +d(t)+\int_0^t \sigma(t,s)(1+\Vert x(s)\Vert )ds\\
        &\leq \vert \dot{\zeta}(t)\vert+L\varrho(t) +c(t)R(t) +d(t)+\int_0^t \sigma(t,s)(1+R(s) )ds\\
        &\leq   \vert \dot{\zeta}(t)\vert +L\rho(t)+\eta(t)R(t)+\varepsilon(t)\\
    &=\nu(t):=\frac{\vert \dot{\zeta}(t)\vert}{1-L}+\frac{1+L}{1-L}(\eta(t)R(t)+\varepsilon(t)).
    \end{aligned}
\end{equation*}
Finally, since the sets $C(t,x(t))$ are uniformly subsmooth and \eqref{sweeping-y} holds, from \eqref{normally-dist}, we obtain that for a.e. $t\in [0,T]$ one has
\begin{equation*}
\begin{aligned}
\frac{\dot{x}(t)-z(t)}{\nu(t)}&\in -N(C(t,x(t));x(t))\cap \mathbb{B}=-\partial d_{C(t,x(t))}(x(t)),
\end{aligned}
\end{equation*}
which proves the result.
\end{proof}

\section{Optimal control of Volterra Sweeping Processes}\label{Control-problem}
In this section, we study the existence of solutions for an optimal control problem governed by state-dependent integro-differential sweeping processes in finite-dimensional spaces. We prove the existence of solutions for the controlled problem without any convexity assumptions on the integral part of the dynamics.

Given a running cost function $\varphi\colon [0,T]\times \mathbb{R}^{2s}\times \mathbb{R}^m\to \overline{\mathbb{R}}$ and a terminal cost function $\ell \colon \mathbb{R}^{2s}\to \overline{\mathbb{R}}$, we consider the following optimal control problem:
\begin{equation}\label{contro_probl}
    \begin{aligned} 
        \min \quad  \ell(x(0),x(T))+\int_{0}^T \varphi(s,x(s),\dot{x}(s),u(s))ds,
    \end{aligned}
\end{equation}
subject to 
\begin{equation*}
\left\{
    \begin{aligned} 
       &\dot{x}(t)\in -N_{C(t,x(t)}(x(t))+F(t,x(t))+\int_{0}^{t} g(t,s,x(s),u(s))ds \textrm{ a.e. } t\in [0,T],\\
        & x(0)\in C(0,x_0),\, u(t) \in U  \textrm{ a.e. } t\in [0,T] \quad 
         (x(0),x(T)) \in S,\\
    \end{aligned} \right.
\end{equation*}
where $u\colon [0,T]\to \mathbb{R}^m$ is a measurable control function, $U\subset \mathbb{R}^m$ is a bounded, closed and  convex set
 and $S$ is a compact set of $\mathbb{R}^{2s}$. Here $\varphi$ is a normal integrand (see, e.g., \cite[Definition~14.27]{MR1491362}).

The following result asserts the existence of optimal controls for the optimization problem \eqref{contro_probl}. Its main feature is that convexity in the integral part of the dynamics is not required for the existence of solutions.
\begin{theorem} Suppose, in addition to \ref{H1F}, \ref{H2F}, and    \ref{H3Fs}, that \ref{Hx1}, \ref{Hx2}, and the following assumptions hold:
    \begin{itemize}
        \item[(i)]   For every $u\in U$, the map $(t,s,x)\mapsto g(t,s,x, u)$ satisfies \ref{Hg}.
\item[(ii)] For all $t\in [0,T]$ the integral mapping \begin{equation}\label{integralmapping}
    	\mathcal{G}(x,u)(t):=\int_{0}^{t} {g(t,s,x(s),u(s))}ds
\end{equation} 
is norm-weak continuous, i.e., 
 $\lim_{n\to +\infty}\mathcal{G}(x_n,u_n)(t)=\mathcal{G}(x,u)(t)$ for all $t\in [0,T]$,  whenever $x_n\to x$ in $C([0,T];\mathbb{R}^s)$ and $u_n \rightharpoonup u$ in $L^2([0,T];\mathbb{R}^s)$.
\item[(iii)] The terminal cost function $\ell$  is lsc.
\item[(iv)] The running cost function $\varphi$ is a normal integrand such that  $\varphi(s,x, \cdot,\cdot)$ is convex for all $(s,x) \in [0,T]\times \mathbb{R}^{s}$. Moreover,  there exists $l \in L^1([0,T],\mathbb{R})$ such that 
    \begin{equation*}
    	\varphi(s,x,y,u) \geq  l(s),\text{ for all } (s,x,y,u) \in [0,T]\times \mathbb{R}^{2s}\times U.
    \end{equation*}
    \end{itemize}
    Then, if the value of the optimal control problem \eqref{contro_probl} is finite, it admits at least one optimal solution.
\end{theorem}
\begin{proof}
 Let us consider a minimizing sequence  $(x_k, u_k) $ of the optimization problem \eqref{contro_probl}. Then, by virtue \ref{H3Fs} and Proposition \ref{reduction-thm}, it follows that  $x_k(t)\in C(t,x_k(t))$ for all $t\in I $ and 
 {\small
 \begin{equation*}
\left\{
\begin{aligned}
&\dot{x}_k(t)\in -\nu(t)\partial  d_{C(t,x_k(t))}(x_k(t))+F(t,x_k(t))+\int_0^t g(t,s,x_k(s))ds  \textrm{ a.e. } t\in I;\\
&x_k(0)=x_0\in C(0,x_0),
\end{aligned}
\right.
\end{equation*}}
\hspace{-1.5mm }for some $\nu\in L^1([0,T];\mathbb{R}_+)$. Hence, we can obtain that the sequence $(x_k)$ satisfies condition \eqref{acotamiento}.   Moreover, by the compactness of $S$ we can assume that $(x_k(0),x_k(T) ) \to (x_0,x_T) \in S$. Then,  we can suppose (up to a subsequence) that $x_k$ satisfies the conclusions of   Lemma \ref{compactness} for some absolutely continuous function $x(\cdot)$ with $ (x(0),x(T) )= (x_0,x_T) $ and 
 \begin{equation}\label{inclusion-f}
 x(t)\in C(t,x(t)) \textrm{ for all } t\in [0,T].
 \end{equation}
 Furthermore, since the set $U$ is bounded, the sequence $(u_k)$ is  bounded in $L^2([0,T];\mathbb{R}^m)$. Thus, we can assume that $u_k \rightharpoonup u\in L^2([0,T];\mathbb{R}^m)$. Hence, by using similar arguments to those given in \eqref{estimation001}, we obtain that  $(x,u)$ satisfies the differential inclusion \eqref{reducido-dif-eq}. Then, again by Proposition \ref{reduction-thm} and inclusion \eqref{inclusion-f}, we obtain that $x(\cdot)$ solves \eqref{Sweeping-Dif}. \newline 
  Finally, by the lower semicontinuity of $\ell$ and the  lower semicontinuity result from \cite[Theorem 2.1]{bal}, we conclude that $(x,u)$ attains the minimum on the control problem \eqref{contro_probl}, which ends the proof.
\end{proof}
The following result provides an example where the integral mapping \eqref{integralmapping}  is norm-weak continuous. 
\begin{example}[Linear control variable]
Let us consider the map $g(s,x,u):=A(s,x)u+ b(s,x)$, where  $A\colon [0,T] \times \mathbb{R}^s \to \mathcal{M}_{n\times m} $, $b\colon [0,T] \times \mathbb{R}^s \to  \mathbb{R}^s$ are continuous functions. Then, the integral mapping \eqref{integralmapping} is norm-weak continuous.  
\end{example}

%\backmatter

\noindent \textbf{Acknowledgments}\\
The authors were supported by ANID Chile under grants Fondecyt Regular  N$^{\circ}$ 1240120 (P. P\'erez-Aros and E. Vilches), Fondecyt Regular N$^{\circ}$ 1220886 (P. P\'erez-Aros and E. Vilches), Fondecyt Regular N$^{\circ}$ 1240335 (P. P\'erez-Aros),  Proyecto de Exploraci\'on 13220097 (P. P\'erez-Aros and E. Vilches),  CMM BASAL funds for Center of Excellence FB210005 (P. P\'erez-Aros and E. Vilches), Project ECOS230027 (P. P\'erez-Aros and E. Vilches), MATH-AMSUD 23-MATH-17 (P. P\'erez-Aros and E. Vilches).

\end{document}